\documentclass[12pt,reqno,a4paper]{amsart} 

\usepackage{a4wide}
\usepackage{latexsym}
\usepackage{indentfirst}
\usepackage{graphicx}
\usepackage{amsmath}
\usepackage{amssymb}
\usepackage{amsthm}
\usepackage{bm}
\usepackage{accents}
\usepackage{enumitem}
\usepackage{relsize}
\usepackage{datetime}
\usepackage[toc,page]{appendix}
\usepackage{graphicx}
\usepackage{setspace}
\usepackage{bigints}

\let\oldproofname=\proofname
\renewcommand{\proofname}{\rm\bf{\oldproofname}}

\makeatletter
\newcommand{\mylabel}[2]{#2\def\@currentlabel{#2}\label{#1}}
\makeatother

\newcommand{\vertiii}[1]{{\left\vert\kern-0.25ex\left\vert\kern-0.25ex\left\vert #1 
    \right\vert\kern-0.25ex\right\vert\kern-0.25ex\right\vert}}

\usepackage{hyperref}
\hypersetup{pdftex,colorlinks=true,linkcolor=red,urlcolor=blue,citecolor=blue}
\usepackage{hypcap}

\newcommand*{\medcap}{\mathbin{\scalebox{1.5}{\ensuremath{\cap}}}}%

\makeatletter
\newcommand{\doublewidetilde}[1]{{%
  \mathpalette\double@widetilde{#1}%
}}
\newcommand{\double@widetilde}[2]{%
  \sbox\z@{$\m@th#1\widetilde{#2}$}%
  \ht\z@=.9\ht\z@
  \widetilde{\box\z@}%
}

\makeatother
 \makeatletter
\newcommand{\mathleft}{\@fleqntrue\@mathmargin50pt}
\newcommand{\mathcenter}{\@fleqnfalse}
\makeatother

\makeatother
 \makeatletter
\newcommand{\zeromathleft}{\@fleqntrue\@mathmargin0pt}
\makeatother

\newtheorem{prop}{Proposition}[section]
\newtheorem*{defi*}{Definition}
\newtheorem{defi}[prop]{Definition}

\newtheorem{lem}[prop]{Lemma}
\newtheorem*{lem*}{Lemma}

\newtheorem{thm}[prop]{Theorem}

\newtheorem{theorem}{Theorem}

\newtheorem{proposition}{Corollary A}

\def\rr{\mathbb{R}}

\numberwithin{equation}{section}

\newcommand*{\corrected}{\textcolor{black}}

\begin{document}


\title[Schr\"{o}dinger equation on star-graphs]{The Schr\"odinger equation on a star-shaped graph under general coupling conditions}

\author[A. Grecu]{Andreea Grecu}
\address[A. Grecu]{University of Bucharest, Academiei Street, No. 14, 010014, Bucharest, Romania \\
\hfill\break\indent \and
\hfill\break\indent
Institute of Mathematics ``Simion Stoilow'' of the Romanian Academy,
Centre Francophone en Math\'{e}matique
\\21 Calea Grivitei Street \\010702 Bucharest \\ Romania}
\email{andreea.grecu@my.fmi.unibuc.ro}

\author[L. I. Ignat]{Liviu I. Ignat}
\address[L. I. Ignat]{Institute of Mathematics ``Simion Stoilow'' of the Romanian Academy,
Centre Francophone en Math\'{e}matique
\\21 Calea Grivitei Street \\010702 Bucharest \\ Romania
}
\email{liviu.ignat@gmail.com}

\begin{abstract}
We investigate dispersive and Strichartz estimates for the Schr\"{o}dinger time evolution propagator $\mathrm{e}^{-\mathrm{i}tH}$ on a star-shaped metric graph. The linear operator, $H$, taken into consideration is the self-adjoint extension of the Laplacian, subject to a wide class of coupling conditions. The study relies on an explicit spectral representation of the solution in terms of the resolvent kernel which is further analyzed using results from oscillatory integrals. As an application, we obtain the global well-posedness for a class of semilinear  Schr\"{o}dinger equations.
\end{abstract}

\maketitle




\section{Introduction}\label{section 1}
It is well-known that in the case of the linear Schr\"{o}dinger equation on $\mathbb{R}$
\begin{equation}\label{system}
\left\{
\begin{array}{lll}
\mathrm{i} u_{t}(t,x)+\Delta u (t,x)=0, &t \neq 0, &x \in \mathbb{R},\\[5pt]
u(0,x)=u_{0}(x), & &x \in \mathbb{R}, 
\end{array}
\right.
\end{equation}
 the unitary group, $(\mathrm{e}^{\mathrm{i}t\Delta})_{t \in \mathbb{R}}$, possesses two important properties which can be derived for instance via Fourier transform (e.g. \cite[Theorem~IX.30]{reedsimon2}): the conservation of the $L^2-$ norm
$$\big\| \mathrm{e}^{\mathrm{i}t\Delta} u_0 \big\|_{L^2(\mathbb{R})} = \| u_0 \|_{L^2(\mathbb{R})}, \quad t \in \mathbb{R},$$
and the dispersion property
$$\big \|\mathrm{e}^{\mathrm{i}t\Delta} u_0\big\|_{L^{\infty}(\mathbb{R})} \lesssim \dfrac{1}{\sqrt{|t|}} \|u_0 \|_{L^1(\mathbb{R})}, \quad t \neq 0.$$
Once the previous inequalities are obtained, the following holds by interpolation for all $ p \in [1,2]$:
\begin{enumerate}[label=(\roman*)]
\item Dispersive estimates:
\begin{align*}
\big\| \mathrm{e}^{\mathrm{i} t \Delta} u_0\big\|_{ L^{p'}(\mathbb{R})} &\lesssim |t|^{\frac{1}{2}-\frac{1}{p}}\|u_0\|_{L^p(\mathbb{R})},\ t\neq 0.
\end{align*}
\end{enumerate}
Moreover, by a classical result of Keel and Tao in \cite{keel} more general space-time estimates, known as Strichartz estimates, also hold:

\begin{enumerate}[resume*]
\item Strichartz estimates (homogeneous, dual \corrected{homogeneous} and inhomogeneous):
\begin{equation*}
\big\| \mathrm{e}^{\mathrm{i} t \Delta} u_0 \big\|_{L^q_t(\mathbb{R},L^{r}_x(\mathbb{R}))} \lesssim \big\|u_0 \big\|_{L^2(\mathbb{R})},
\end{equation*}
\begin{equation*}
\bigg\| \int_{\mathbb{R}} \mathrm{e}^{-\mathrm{i} s \Delta} F(s,\cdot) ds \bigg\|_{L^2(\mathbb{R})} \lesssim \big\|F \big\|_{L^{q'}_t(\mathbb{R}, L^{r'}_x(\mathbb{R}))},
\end{equation*}
\begin{equation*}
\bigg\| \int_{s < t} \mathrm{e}^{\mathrm{i} (t-s) \Delta} F(s,\cdot) ds \bigg\|_{L^q_t(\mathbb{R}, L^{r}_x(\mathbb{R}))} \lesssim ||F ||_{L^{\tilde{q}'}_t(\mathbb{R}, L^{\tilde{r}'}_x(\mathbb{R}))},
\end{equation*}
\end{enumerate}
where $(q,r)$ and $(\tilde{q},\tilde{r})$ are sharp ${1}/{2}$--admissible exponent pairs. We recall that $(q,r)$ is ${1}/{2}$--admissible if $2 \leq q, r \leq \infty$ and satisfy the relation
\begin{equation}
\dfrac{1}{{q}} =\frac{1}{2}\Big(\frac {1}{2} - \frac{1}{r}\Big). \label{onehalf}
\end{equation}
These estimates are key ingredients in order to obtain well-posedness   of the nonlinear Schr\"{o}dinger equation for a class of power nonlinearities (e.g. \cite{tsutsumi, cazenave, linares}).\\

In this paper we derive analogous dispersive and Strichartz estimates for the Schr\"{o}dinger time evolution propagator $\mathrm{e}^{\mathrm{i}t \Delta (A,B)} \mathcal{P}_{ac}(-\Delta(A,B))$ in the case of a star-shaped metric graph, with general boundary conditions at the common vertex which induce a self-adjoint extension of the Laplacian $\Delta(A,B)$, where $\mathcal{P}_{ac}(-\Delta(A,B))$ denotes the projection onto the absolutely continuous spectral subspace of $L^2(\mathcal{G})$ associated to $-\Delta(A,B)$. The considered star-graph consists of a finite number of edges of infinite length attached to a common vertex, each of them being identified with a copy of the positive semi-axis. We point out that in this context we cannot use Fourier analysis to obtain the dispersive properties, as in the case of the real line. In order to overcome this impediment we adopt a spectral theoretical approach and we use some tools specific to oscillatory integrals.

It is worth to mention that in the case of a star-shaped metric graph these results were established in \cite{adami} only for particular cases of boundary conditions, namely Kirchhoff, $\delta$ and $\delta'$ coupling conditions. So, this makes our result, to the best of our knowledge, new. Also for star-graphs, in the case of Kirchhoff boundary conditions, similar results were obtained in \cite{mehmeti} where the Schr\"{o}dinger group $\mathrm{e}^{\mathrm{i}t(-\Delta + V)}$ is considered for real valued potentials $V$ satisfying some regularity and decay assumptions. We mention here that the latter is still in open problem in the case of more general boundary conditions. B\u{a}nic\u{a} and Ignat proved the same results in the case of trees with the last generation of edges of infinite length, with Kirchhoff coupling condition at the vertices, in the joint work \cite{banica2011} and in \cite{ignatsiam}. With the same conditions, dispersive estimates were obtained in the case of the tadpole graph in \cite{mehmetitadpole}. The extension of the results in this paper to general trees as in \cite{banica2011,banica2014} remains to be investigated, but new ideas have to be used.

The paper is organized as follows: in Section  \ref{section 2}, we give some preliminaries on Laplacians on metric star-graphs. We start by fixing some notations, we define the function spaces, we introduce the Laplace operator on star-graph, moreover, we recall a general result of Kostrykin and Schrader \cite{laplacians} that provides necessary and sufficient conditions for the self-adjointness of the Laplacian. Then we introduce the Schr\"{o}dinger equation on a star-shaped metric graph \corrected{$\mathcal{G}$} and we state the main results of this paper, Theorem \ref{theorem 1} concerning the dispersive and Strichartz estimates,  Theorem \ref{theorem 2} \corrected{and Theorem \ref{theorem 3}}, well-posedness results of the nonlinear equation
\corrected{
\begin{equation} \label{system nonlinear}
\left\{
\begin{array}{lll}
\mathrm{i} u_{t}(t,x)+\Delta(A,B) u (t,x) +  \corrected{ \lambda |u(t,x)|^{p-1} u(t,x)} =0, &t \neq 0,  &x\in \mathcal{G},\\[5pt]
u(0,x)=u_{0}(x), & &x\in \mathcal{G},
\end{array}
\right.
\end{equation}}
\corrected{In both cases, we consider $1<p<5$. Theorem  \ref{theorem 2} treats the case of $L^2(\mathcal{G})$-solutions, while Theorem \ref{theorem 3} considers solutions in the energy space $D(\mathcal{E})$.}

In Section \ref{section 3} we give the proofs of the main results following several steps. First we localize the absolutely continuous spectrum of the operator and then we give an explicit description of the solution of system \eqref{system star} using spectral theory tools. The dispersion properties are obtained based on this explicit form, relying on tools specific to oscillatory integrals, and the Strichartz space-time estimates follow as soon as the previously mentioned properties are proven. Finally, we prove the well-posedness in $L^2(\mathcal{G})$  \corrected{and in the form domain $D(\mathcal{E})$}   of the semi-linear Schr\"{o}dinger equation for a class of  nonlinearities in the sub-critical case (the terminology in this framework is described in detail in \cite{cacciapuoti2017}).

\section{Preliminaries and main results} \label{section 2}

\corrected{
We divide this section on two parts. In the first one we present general results about metric graphs and self-adjoint Laplace operators on such structure. In the second part we state the main results of this paper and comment on the previous work and main difficulties that arise.}

\subsection{\corrected{Self-adjoint Laplacians on graphs}} In the following we collect a few results necessary for a self-contained presentation as possible. For further details we refer to \cite{berkolaiko, olaf} and references therein.

\noindent \begin{defi} \corrected{ A discrete graph $\mathcal{G}:=(V,E,\partial)$ consists of a finite or countably infinite set of vertices $V=\{v_i\}$, a set of adjacent edges at the vertices $E=\{e_j\}$, internal and/or external, and an orientation map $\partial : E \to V \times V$ which associates to each internal $e_j$ edge the pair $(\partial_{-}e_j,\partial_{+}e_j)$ of its initial and terminal vertex, and to an external edge its initial vertex only.}

\end{defi}

\noindent \corrected{Each internal edge $e_j$ of the graph can be identified with a finite segment $I_j=[0,l_j] $ of the real line, such that $0$ correspond to its initial vertex and $l_j$ to its terminal one; each external edge $e_j$, with the half line $[0,\infty)$, with $0$ corresponding to its initial vertex. This defines a natural topology on $\mathcal{G}$ (the space of union of all edges).}

\begin{defi} \corrected{A metric graph is a discrete graph together with the set of edge lengths $\{ l_j \}_{j}$, equipped with a natural metric, with the distance of two points to be the length of the shortest path in $\mathcal{G}$ linking the points.}
\end{defi}

In the sequel, we consider a metric graph $\mathcal{G}$ as below, given by a finite number $n \in \mathbb{N}^{*}, n \geq 3$ of infinite length edges attached to a common vertex (a so called star-shaped metric graph), having each  edge identified with a copy of the half-line $[0,\infty)$. A function $u$ defined on $\mathcal{G}$ is a vector $u=(u_1,\dots ,u_n)^T$, each $u_j$ being defined on the interval $I_j=[0,\infty)$, $j=1,\dots, n$. 
\begin{center}
\includegraphics[scale=0.14]{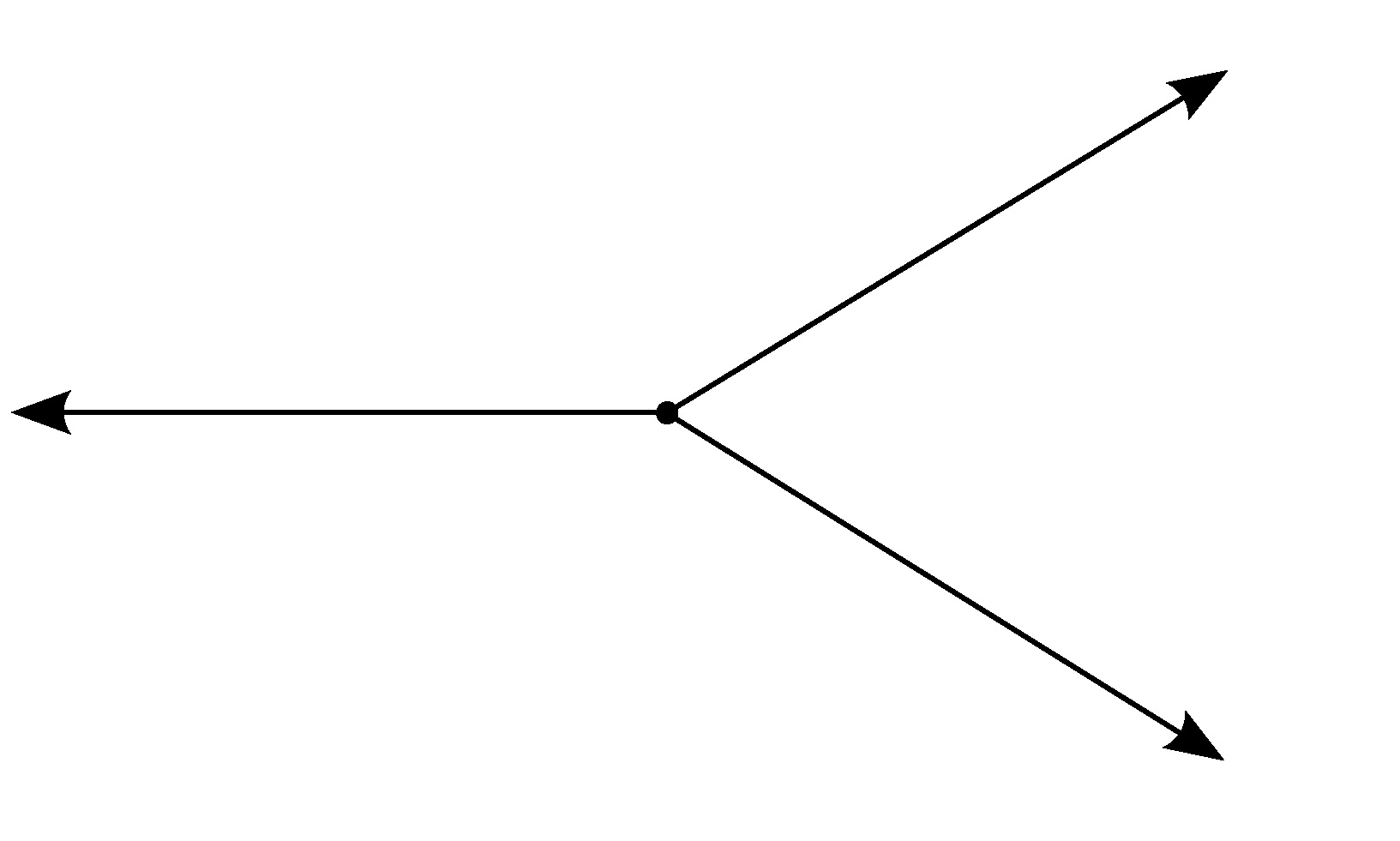}\\
{\textbf{Figure 1}: \textit{The star-graph $\mathcal{G}$ (one vertex, \corrected{$n=3$} edges)}}
\end{center}
The Lebesgue measures on the intervals $I_j $ induce a Lebesgue measure on the space $\mathcal{G}$. We introduce the Hilbert space $L^2(\mathcal{G})$ as the space of measurable and square-integrable functions on each edge of $\mathcal{G}$, i.e.
$$ L^2(\mathcal{G}) = \bigoplus_{j=1}^{n} L^2(I_j), \quad \| f \|^2_{L^2(\mathcal{G})} = \sum_{j=1}^{n} \int_{I_j} |f_j(s)|^2 ds,$$
 with $f=(f_j)^T_{j=\overline{1,n}}$, where $f_j \in L^2(I_j )$ is a complex valued function. The inner product $\langle \cdot, \cdot \rangle$ is the one induced by the usual inner product in $L^2([0,\infty))$, i.e.
\begin{equation*}
\langle f, g \rangle = \sum_{j=1}^{n} \int_{I_j} f_j(s) \overline{g_j}(s) \; ds.
\end{equation*}
Analogously, given $1 \leq p \leq\infty$ one can define $L^p(\mathcal{G})$ as the set of functions whose components are elements of $L^p(I_j)$, and the corresponding norm defined by 
$$ \| f \|^p_{L^p(\mathcal{G})} = \sum_{j=1}^{n} \| f \|^p_{L^p(I_j)} \quad for \quad 1 \leq p < \infty, \quad and   \quad  \| f \|_{L^{\infty}(\mathcal{G})} = \sup_{1 \leq j \leq n} \| f \|_{L^{\infty}(I_j)}. $$
 In order to define Laplacian operators, we consider the Sobolev space \corrected{$ H^m(\mathcal{G})$, $m\in \{1,2\}$},
$$ \corrected{H^m(\mathcal{G}) = \bigoplus_{j=1}^{n} H^m(I_j), \quad \| f \|^2_{H^m(\mathcal{G})} = \sum_{j=1}^{n} \| f \|^2_{H^m(I_j)} ,}$$
where \corrected{$H^m(I_j)$} is the classical Sobolev space on $I_j$, $j=1,\dots,n$. \corrected{We emphasize that in this definition we are not assuming any condition on the values of the functions at the joint point.}  

In the sequel we introduce the Laplacian  $\Delta(A,B)$ with the domain 
\begin{equation}
\mathcal{D}(\Delta (A,B))=\{ u \in H^2(\mathcal{G}) : A \underline{u} + B \underline{u'}=0 \} , \label{domain}
\end{equation}
acting as the second derivative along the edges, 
\[
\Delta (A,B)u = (u''_{1}, u''_{2},\dots,u''_{n})^{T}.
\]
In the definition of the domain,  $A$ and $B$ are $\mathbb{C}^{n \times n}$ matrices which express the coupling condition at the common vertex, and $\underline{u}=(u_1(0),\dots, u_n(0) )^T$ , $\underline{u'}=(u'_1(0_+), \dots, u'_n(0_+))^T$, respectively.

A crucial result concerning the parametrization of all self-adjoint extensions of the Laplace operator in $L^2(\mathcal{G})$ in terms of the boundary conditions, due to \cite{laplacians}, is the following: for any two $n \times n$ matrices $A$ and $B$, the next two assertions are equivalent:

\begin{enumerate}[label=(\roman*)]
\setlength\itemsep{2pt}

\item The operator $\Delta(A,B)$ is self-adjoint;

\item $A$ and $B$ satisfy:
\vspace{2pt}

\begin{enumerate}
\setlength\itemsep{2pt}

\item[\mylabel{H1}{(H1)}]The horizontally concatenated matrix $(A,B)$ has maximal rank;

\item[\mylabel{H2}{(H2)}] $AB^\dag$ is self-adjoint, i.e. $AB^\dag=BA^\dag$.

\end{enumerate}

\end{enumerate}

\noindent Matrices $A$ and $B$ satisfying \ref{H1} and \ref{H2} may also appear in the literature under the name of {\it Nevanlinna} pair (e.g. \cite{behrndt}). The most used type of couplings that satisfy the above hypotheses \ref{H1} and \ref{H2} are the following ones:
%
\begin{enumerate}
\item {\it {Kirchhoff--coupling:}}
\mathleft
\[
\psi_i(0)=\psi_j(0), \quad i,j=1,\dots,n, \qquad {\sum_{j=1}^{n}} \psi'_j(0_+)=0 ; \qquad \qquad  \quad 
 \]
\item {\it {$\delta$--coupling:}}
\[
\psi_i(0)=\psi_j(0), \quad i,j=1,\dots,n, \qquad {\sum_{j=1}^{n}} \psi'_j(0_+)=\alpha \psi_k(0), \ \alpha \in \mathbb{R}; 
\]
\item {\it {$\delta'$--coupling:}}
\[
\psi'_i(0_+)=\psi'_j(0_+), \quad i,j=1,\dots,n, \quad {\sum_{j=1}^{n}} \psi_j(0)=\beta \psi'_k(0_+), \ \beta \in \mathbb{R}.
 \]
\end{enumerate}

\noindent \corrected{For instance, in the case of \textit{Kirchhoff coupling}, the associated matrices $A$ and $B$ can be written as}
\corrected{
\[
A = \begin{bmatrix} 
    1 & -1 & 0 & \dots & 0 & 0 \\
    0 & 1 & -1 & \dots & 0 & 0 \\
    \vdots & \vdots & \vdots &  & \vdots & \vdots\\
    0 &  0  & 0 & \dots & 1 & -1\\
    0 &  0  & 0 & \dots & 0 & 0 
    \end{bmatrix}
, \quad 
B = \begin{bmatrix} 
    0 & 0 & 0 & \dots & 0 & 0 \\
    0 & 0 & 0 & \dots & 0 & 0 \\
    \vdots & \vdots & \vdots &  & \vdots & \vdots\\
    0 &  0  & 0 & \dots & 0 & 0\\
    1 &  1  & 1 & \dots & 1 & 1 
    \end{bmatrix}.
\]
}
\corrected{Note that this representation is not unique. More precisely, the Laplacians $\Delta(A,B)$ and $\Delta(A',B')$ coincide for equivalent pairs $(A,B)$ and $(A',B')$, i.e. they satisfy \ref{H1}, \ref{H2} and exists an invertible matrix $C$ such that $A'=CA$ and $B'=CB$.}\\

\corrected{According to \cite[Theorem 1.4.4]{berkolaiko}, the above conditions \ref{H1}, \ref{H2} on $A$ and $B$ can be described in an equivalent way. More precisely, 
the operator $\Delta$ on star-graph is self-adjoint if and only if there are three orthogonal (an mutually orthogonal) projectors, $P_{D}, P_{N}$ and $P_{R}=I-P_{D}-P_{N}$ acting on $\mathbb{C}^{n}$ and $\Lambda$ acting on $P_{D}\mathbb{C}^{n}$, invertible and self-adjoint, such that the boundary values $\underline{f}$ of $f$ satisfy
\begin{equation*}
\begin{cases}
P_{D} \underline{f}=0, & "Dirichlet \, Part"\\
P_{N}\underline{f}=0, & "Neumann \, Part"\\
P_{R}\underline{f'}=\Lambda P_{R}\underline{f}, & "Robin \, Part"\\
\end{cases}.
\end{equation*}}
\corrected{In fact, $P_D$ and $P_N$ are the projections on the kernel of $B$, respectively $A$ and $P_R=I-P_D-P_N$.
Explicit representations of these projectors in the particular cases described above can be found in \cite[Section 1.4.4, p. 24]{berkolaiko}.}

\corrected{ Following \cite[Theorem 1.4.11, p.22]{berkolaiko}, the quadratic form associated to $-\Delta(A,B)$ can be written as
\begin{equation}\label{quadraticform}
\begin{aligned}
\mathcal{E}(f,f) &= \sum_{j=1}^{n} \int_{I_j}  |f'_j(x)|^2 \, dx + < \Lambda P_{R} \underline{f}, P_{R} \underline{f} >_{\mathbb{C}^n},
\end{aligned}
\end{equation}}

\noindent \corrected{for any $f$ in the form domain
\begin{equation}\label{formdomain}
D(\mathcal{E})= \{  f \in H^1(\mathcal{G}) \, : \, P_{D}\underline{f}=0 \},
\end{equation}}

\noindent \corrected{and the corresponding sesqui-linear form
\begin{equation*}
\mathcal{E}(f,g)= \sum_{j=1}^{n} \int_{I_j}  f'_j(x) \overline{g}'_j(x) \, dx + < \Lambda P_{R} \underline{f}, P_{R} \underline{g} >_{\mathbb{C}^n}.
\end{equation*}}
\corrected{Moreover, for $M$ sufficiently large, the norm
\begin{equation}\label{formnorm}
\| f \|_{D(\mathcal{E})} := \sqrt{M \| f \|_{L^2(\Gamma)} + \mathcal{E}(f,f)}, \ f \in D(\mathcal{E})
\end{equation}}
\corrected{is equivalent  \cite[p.23]{berkolaiko}  to the norm of the space $H^1(\Gamma)$, i.e. exist $c_1, c_2 > 0$ such that
\begin{equation}\label{equivnorms}
c_1 \| f \|_{H^1(\Gamma)} \leq \| f \|_{D(\mathcal{E})} \leq c_2 \| f \|_{H^1(\Gamma)}.
\end{equation}}
\corrected{We denote by $D(\mathcal{E})^*$ the dual of the energy space $D(\mathcal{E})$. }

The subject of self-adjoint Laplacians on metric graphs has become popular under the name of "quantum of graphs", known for its wide applications in quantum mechanics. Since the dynamics of a quantum system is described by unitary operators, most of the literature has been concerned with self-adjoint operators. However, in \corrected{\cite{nonselfadjoint,husseinmaximal}}, non-self-adjoint Laplacians on graphs were considered, since there are cases where a system is described by non-conservative equations of motion. \corrected{As also mentioned in the previous works, many non-self-adjoint operators that satisfy the so called $\mathcal{PT}$--symmetry property, posses real spectrum. This is a necessary condition in order for the operator to be quantum-mechanically admissible. In the case of star-graphs, in \cite{kurasovrtsymmetry} such Laplace operators are studied, and a description in terms of the boundary conditions is provided.} Also, in \cite{kostrykincontraction}, contraction semigroups which are generated by Laplace operators which are not necessarily self-adjoint are studied. They provide a criteria for such semigroups to be continuity and positivity preserving and a  characterization of generators of Feller semigroups on metric graphs.

\subsection{\corrected{Main results}}
The model we will consider in this paper is given by the following linear Schr\"{o}dinger equation
\begin{equation}\label{system star}
\left\{
\begin{array}{lll}
\mathrm{i} u_{t}(t,x)+\Delta(A,B) u (t,x)=0, &t \neq 0, &x\in \mathcal{G}, \\[5pt]
u(0,x)=u_{0}(x), & &x\in \mathcal{G}, 
\end{array}
\right. 
\end{equation}
where matrices $A$ and $B$ are assumed to satisfy \ref{H1} and \ref{H2}. In the following, $\mathrm{e}^{\mathrm{i} t \Delta(A,B)}$ is the unitary group generated by $\Delta(A,B)$ with domain \eqref{domain}, $\mathcal{P}_{ac}(-\Delta(A,B))$ denotes the projection onto the absolutely continuous spectral subspace of $L^2(\mathcal{G})$ associated to $-\Delta(A,B)$. Along the paper $(q,r)$ and $(\tilde{q},\tilde{r})$ are sharp $1/2$--admissible exponent pairs, i.e. \eqref{onehalf} holds true, unless we specify otherwise
and $p'$ stands for the dual exponent of $p$. Also in order to simplify the presentation, by writing $f\lesssim g$ we understand that the inequality holds up to some positive multiplicative constants that depends on the matrices $A$ and $B$ and on the involved exponents.  

We will give now the statements of the main results of our paper.

\begin{theorem}\label{theorem 1} The time-evolution propagator $\mathrm{e}^{\mathrm{i} t \Delta(A,B)} \mathcal{P}_{ac}(-\Delta(A,B))$ satisfies 
\begin{enumerate}[label=(\roman*)]
\item Dispersive estimates: for all $p\in [1,2]$,
\begin{equation}\label{estimate 1}
\big\| \mathrm{e}^{\mathrm{i} t \Delta(A,B)} \mathcal{P}_{ac}(-\Delta(A,B)) u_0 \big\|_{ L^{p'}(\mathcal{G})} \leq C(A,B,p) |t|^{\frac{1}{2}-\frac{1}{p}}\|u_0\|_{L^p(\mathcal{G}) }, \forall \ t \neq 0;
\end{equation}

\item Strichartz estimates (homogeneous, dual homogenous and inhomogeneous): \label{strichartz}
\mathleft
\begin{equation*}\label{estimate 2}
\big\| \mathrm{e}^{\mathrm{i} t \Delta(A,B)} \mathcal{P}_{ac}(-\Delta(A,B)) u_0 \big\|_{L^q(\mathbb{R}, L^{r}(\mathcal{G}))} \lesssim \big\|u_0 \big\|_{L^2(\mathcal{G})},
\end{equation*}
\begin{equation*}\label{estimate 3}
\bigg\| \int_{\mathbb{R}} \mathrm{e}^{-\mathrm{i} s \Delta(A,B)} \mathcal{P}_{ac}(-\Delta(A,B)) F(s,\cdot) ds \bigg\|_{L^2(\mathcal{G})} \lesssim \big\|F \big\|_{L^{q'}(\mathbb{R}, L^{r'}(\mathcal{G}))},
\end{equation*}
\begin{equation*}\label{estimate 4}
\bigg\| \int_{s < t} \mathrm{e}^{\mathrm{i} (t-s) \Delta(A,B)} \mathcal{P}_{ac}(-\Delta(A,B)) F(s,\cdot) ds \bigg\|_{L^q(\mathbb{R}, L^{r}(\mathcal{G}))} \lesssim ||F ||_{L^{\tilde{q}'}(\mathbb{R}, L^{\tilde{r}'}(\mathcal{G}))}.
\end{equation*}
\end{enumerate}
\end{theorem}
In all the estimates above one requires the projection onto the absolutely \corrected{continuous} spectrum as the Laplace operator $-\Delta(A,B)$ on $L^2(\mathcal{G})$ may posses non-empty point spectrum, for which the time decay in \eqref{estimate 1} of the unitary group $\mathrm{e}^{\mathrm{i}t \Delta(A,B)}$ cannot occur and then the global estimates in the second part do not hold.
If we need only local in time estimates, as in the case of the proof of the local in time existence of solutions for nonlinear problems, the restriction on $\mathcal{P}_{ac}(-\Delta(A,B)) $ is not required and the following holds true.
\begin{proposition} \label{corollaryA1}
For every $\alpha \geq 1 $, 
\mathleft
\begin{equation*}\label{estimate.2}
\| \mathrm{e}^{\mathrm{i} t \Delta(A,B)} u_0 \|_{L^q((-T,T),L^{r}(\mathcal{G}))} \lesssim \big\|u_0 \big\|_{L^2(\mathcal{G})} +C(\alpha,r) (2T)^{1/q} \big\|u_0 \big\|_{L^{\alpha}(\mathcal{G})}
\end{equation*}
and
\begin{multline*}\label{estimate.4}
\bigg\| \int_0^t \mathrm{e}^{\mathrm{i} (t-s) \Delta(A,B)} F(s,\cdot) ds \bigg\|_{L^q((-T,T),L^{r}(\mathcal{G}))} \lesssim ||F ||_{L^{\tilde{q}'}((-T,T)), L^{\tilde{r}'}(\mathcal{G}))}\\
+ C(\alpha,r)(2T)^{1/q} ||F ||_{L^{1}((-T,T)), L^{\alpha}(\mathcal{G}))}.
\end{multline*}
\end{proposition}
We emphasize that choosing $T<1$, $\alpha=2$ and $\alpha=\tilde r'$ in the first respectively second estimate we obtain local Strichartz estimates similar to the ones in the classical case.

\medskip

 As an application of Corollary A\ref{corollaryA1}, we give the following well-posedness \corrected{results}. We recall that the nonlinearity in \eqref{system nonlinear} is understood component-wise, i.e. $|u \vert^{p-1} u = (
           \vert u_{1} \vert ^{p-1} u_{1},\dots,           \vert u_{n} \vert ^{p-1} u_{n})^T.$
\begin{theorem}\label{theorem 2}
For any $p\in (1,5)$, \corrected{$\lambda \in \mathbb{R}$} and for every $u_0 \in L^2(\mathcal{G})$, there exists a unique mild solution $u \in C(\mathbb{R}, L^2(\mathcal{G})) \underset{(q,r)-adm}{\medcap} L^q_{loc}(\mathbb{R}, L^r(\mathcal{G})) $ of the nonlinear Schr\"{o}dinger equation \eqref{system nonlinear}
where the intersection  is taken over all sharp ${1}/{2}-$admissible exponent pairs.
 Moreover, the $L^2(\mathcal{G})$-norm of $u$ is preserved along time:
$$ \| u(t) \|_{L^2(\mathcal{G})}=\| u_0 \|_{L^2(\mathcal{G})} ,\quad \forall \ t\in \rr.$$
\end{theorem}

\begin{theorem}\label{theorem 3}
\corrected{
For any $p\in (1,5)$, \corrected{$\lambda \in \mathbb{R}$} and for every $u_0$ in the energy domain $ D(\mathcal{E})$, there exists a weak solution $u \in C(\mathbb{R}, D(\mathcal{E})) \cap   C^1(\mathbb{R}, D(\mathcal{E})^*)$ of equation \eqref{system nonlinear}.
 Moreover, the $L^2$--norm and the energy 
$$E(u (t)):= \mathcal{E}(u(t),u(t)) - \dfrac{\lambda}{p+1} \mathlarger\int_{\mathcal{G}} |u (t,x)|^{p+1} \, dx $$ are conserved for all $t \in \mathbb{R}$.}
\end{theorem}

\corrected{
We strongly believe that local well-posedness results can be extended to the whole class of power nonlinearities, $1\leq p<\infty$, following similar arguments as in the one-dimensional case \cite[Chapter 5.2]{linares}.  Concerning the conditions that guarantee the global well-posedness in the energy space, we refer to \cite[Theorem 6.2, p.128]{linares}. }

\corrected{
The main difference between the sub-critical case $p<5$ and the super-critical case $p> 5$ comes from the fact that in the first one we can start with the $L^2$-solution $u$ given by Theorem \ref{theorem 2}   and prove that if the initial data is more regular, i.e. in $D(\mathcal{E})$, then $u$ belongs also to $D(\mathcal{E})$. The main difficulty even in this particular case is due to the fact that the nonlinearity $g(u):=\lambda |u|^{p-1}u$ does not leave invariant the space $D(\mathcal{E})$. In the particular cases considered in \cite{golo,pava2017orbital}, i.e. $\delta$ and $\delta'$-coupling, the energy domain $D(\mathcal{E})$  is preserved by the nonlinearity and then the fix point argument easily applies to obtain the local existence of the mild solutions. However, in the case of general coupling conditions as the ones we are interested in this work,
a function $u\in D(\mathcal{E})$ has the property that $P_D \underline{u}=0$. It is clear that $P_D \underline{u}=0$ does not necessarily imply that $P_D \underline{g(u)}=0$, excepting particular cases.
}
           
The study of nonlinear propagation in ramified structures is of relevance in several branches of pure and applied science, modeling phenomena such as nonlinear electromagnetic pulse propagation in optical fibers, the hydrodynamic flow, electrical signal propagation in the nervous system, etc. For a large classification of applications and references, we refer to   \cite[Chapter 7]{berkolaiko}. We also point out here that the modelization depends on the real network, where each edge has a thickness, but usually idealizations of these graphs are considered. The convergence of these so-called graph-like spaces to metric graphs (with 0-thickness limit) \corrected{ is analyzed in \cite{ olafexner2,kuchmentzeng,molchanov,molchanovconvergence,exnerpost,olaf,rubinstein,smilanskysolomyak,matrasulov}}.\\

\section{Proofs of the main results} \label{section 3}

\subsection{The solution via the resolvent}

\corrected{In the sequel we make use of the explicit representation of the integral resolvent kernel, given in terms of the boundary condition matrices, $A$ and $B$. We recall below the result stated by Kostrykin and Schrader in \cite{laplacians}. Nevertheless, it is worth mentioning that it basically relies on an analog of the classical Krein's formula \cite[Theorem~1.2.1]{albeveriokurasov}, which was obtained both in \cite{albeveriokrein} and in \cite{kostrykinkirchhoff}.}

\begin{lem}\cite[Lemma~4.2]{laplacians} \label{lemma.kernel} The resolvent $(-\Delta(A,B)-\mathrm{k}^2)^{-1}$, for $\mathrm{k}^2 \in \mathbb{C} \setminus \sigma(-\Delta(A,B))$, $\Im (\mathrm{k}) > 0$,  is the integral operator with the $n \times n$ matrix-valued integral kernel $\mathrm{r}(x,y,\mathrm{k})$, admitting the representation
\begin{equation}
\mathrm{r}(x,y,\mathrm{k})=\mathrm{r}^{(0)}(x,y,\mathrm{k}) + \dfrac{\mathrm{i}}{2\mathrm{\mathrm{k}}} \phi(x,\mathrm{k}) G(\mathrm{k},A,B)\phi(y,\mathrm{k}), \label{kernel}
\end{equation}

\noindent with  $[\mathrm{r}^{(0)}(x,y,\mathrm{k})]_{j,j'}=\dfrac{\mathrm{i}}{2\mathrm{k}} \delta_{j,j'} \mathrm{e}^{\mathrm{i} \mathrm{k} |x_j-y_{j'} |}, \ x_j \in I_j, y_{j'} \in I_{j'}, \  \phi(x,\mathrm{k})= \mathrm{diag} \{ \mathrm{e}^{\mathrm{i} \mathrm{k} x_j} \}_{j=\overline{1,n}},\ \phi(y,\mathrm{k})= $ 

\vspace{0.1cm}
\noindent $\mathrm{diag}\{\mathrm{e}^{\mathrm{i} \mathrm{k}  y_j} \}_{j=\overline{1,n}}$ and $ G(\mathrm{k},A,B)=-(A+\mathrm{i} \mathrm{k} B)^{-1}(A-\mathrm{i} \mathrm{k} B).$

\end{lem}

We start by localizing the absolutely continuous spectrum of the self-adjoint Laplacian on the star-graph by using some classical results of Weyl and information concerning the eigenvalues, provided by Kostrykin and Schrader.\\

\begin{prop}

Assume $n \times n$ matrices $A$ and $B$ satisfy hypotheses \ref{H1} and \ref{H2}. Then the absolutely continuous spectrum of the corresponding Hamiltonian $-\Delta(A,B)$ with domain given in \eqref{domain} is the interval $[0, \infty)$.

\end{prop}

We include below the proof of this proposition. However, we would like to mention that we recently encountered a similar result in \cite{schraderabscont}.

\begin{proof}

The operator not being positive definite, the proof is thus not immediate, so we perform a different approach. Since the self-adjoint operator $-\Delta(A,B)$ may have at most a finite collection of negative eigenvalues \cite[Theorem 3.7]{laplacians}, $\sigma_{ac}(-\Delta(A,B))=\sigma_{ess}(-\Delta(A,B))$. Now, the idea is to show first that $\sigma_{ess}(-\Delta(A,B))=\sigma_{ess}(-\Delta(A'=\mathbb{I}_n,B'=\mathbb{O}_n))$, where $-\Delta(A'=\mathbb{I}_n,B'=\mathbb{O}_n)$ is the Hamiltonian with homogeneous Dirichlet boundary conditions, and secondly, that $\sigma_{ess}(-\Delta(A'=\mathbb{I}_n,B'=\mathbb{O}_n))=[0, \infty)$.

Since both Laplacians are self-adjoint (they fulfill \ref{H1} and \ref{H2}), in order to prove the first part, it is enough to show that the resolvent difference $R_{-\Delta(A,B)}-R_{-\Delta(A'=\mathbb{I}_n,B'=\mathbb{O}_n)}$ is a compact operator for some element (and hence, for all) in both resolvent sets, due to Theorem of Weyl in \cite[Theorem~6.19, p.~146]{teschl}. Using each resolvent's representation in terms of its explicit integral kernel expressed in \eqref{kernel}, one can check that the difference of resolvents is a Hilbert Schmidt integral operator for some element in both resolvent sets, and hence, compact (take for instance the spectral parameter $\mathrm{k}^2$ to be any negative real number less than the smallest negative eigenvalue, given in \cite[Remark 3.11]{laplacians}.

For the second part, we will proceed by double inclusion. Since $-\Delta(A'=\mathbb{I}_n,B'=\mathbb{O}_n)$ posses no point spectrum and it is a positive definite self-adjoint operator, $\sigma_{ess}(-\Delta(A'=\mathbb{I}_n,B'=\mathbb{O}_n))=\sigma(-\Delta(A'=\mathbb{I}_n,B'=\mathbb{O}_n))$ and the direct inclusion follows immediately. For the reverse inclusion,  we rely on the criteria of Weyl \cite[Lemma 2.16]{teschl}.
\end{proof}

\noindent {\bf Description of the solution.} By Stone's Theorem in \cite{stone},  which ensures the existence and uniqueness of an unitary group $(\mathrm{e}^{\mathrm{i}t\Delta(A,B)})_{t \in \mathbb{R}}$, the linear Schr\"{o}dinger equation \eqref{system star} is well-posed and the unique global solution is given by $\mathrm{e}^{\mathrm{i}t\Delta(A,B)}u_0$, for all $t \in \mathbb{R}$, and thus, the study of its properties is consistent. In the sequel, we provide an explicit representation of the solution, more precisely, of a sequence of functions converging in $L^2(\mathcal{G})$ to the solution, and hence, almost everywhere along a sub-sequence.\\

A careful analysis of the integral resolvent kernel \eqref{kernel} leads to the following lemma.

\begin{lem}\label{uniform}There exists $\delta_0=\delta_0(A,B)$ such that
\begin{equation}\label{bound_strip}
	\sup_{x,y\in\mathcal{G}, k\in \rr\setminus\{0\},0\leq \delta\leq \delta_0}|\mathrm{k}\mathrm{r}(x,y,\sqrt{\mathrm{k}^2 \pm \mathrm{i} \delta})|<\infty
\end{equation}
Moreover, for any $x,y\in\mathcal{G}$
$$ \lim_{\delta \to 0} \mathrm{k} \mathrm{r}(x,y,\sqrt{\mathrm{k}^2 \pm \mathrm{i} \delta}) = \mathrm{k} \mathrm{r}(x,y,\pm|\mathrm{k}|), \ \forall\, \mathrm{k} \in \rr\setminus\{0\}.$$\
%
%
\end{lem}

\begin{proof}

In view of \cite[Proposition 3.11]{inverse}, $\det(A+\mathrm{i k} B) $ has zeros only on the imaginary axis. Denote by $\rho$ its smallest non-zero root in absolute value. We will chose latter $\delta_0$ such that $\delta_0 \ll \rho$.

By Lemma \ref{lemma.kernel}, for $\mathrm{k} \in \mathbb{R}$ and $j,j'=1,\dots,n$,
\begin{equation}
\label{rep.r}
	 [ \ \mathrm{k} \mathrm{r} (x,y, \sqrt{\mathrm{k}^2 \pm \mathrm{i} \delta}) \ ]_{j,j'} = \frac{\mathrm{i k}}{2 \sqrt{\mathrm{k}^2 \pm \mathrm{i} \delta}} \ \Big[ \delta_{j,j'} \mathrm{e}^{\mathrm{i} \sqrt{\mathrm{k}^2 \pm i \delta} |x_j- y_{j'}|} +\mathrm{e}^{\mathrm{i} \sqrt{\mathrm{k}^2 \pm i \delta} (x_j + y_{j'})} G_{j,j'} (\sqrt{\mathrm{k}^2 \pm \mathrm{i} \delta}, A, B) \Big],
\end{equation}
where $ [ \cdot ]_{j,j'}$ denotes the $j,j'$ entry of the corresponding matrix.
Recall that here the complex square root is chosen in such a way that $\sqrt{ r \mathrm{e}^{\mathrm{i} \theta}} =\sqrt{r} \mathrm{e}^{\mathrm{i} \theta/2}$ with $r>0$ and $\theta \in [0 , 2\pi)$.

We  analyze the elements of the matrix $G(\sqrt{\mathrm{k}^2 \pm \mathrm{i} \delta},A,B)$, making use of the properties of $G(\mathrm{k},A,B)=\corrected{-}(A+\mathrm{ik}B)^{-1}(A-\mathrm{ik}B)$, for $\mathrm{k} \in \mathbb{R} \setminus \{0\}$.
Using the matrix inverse formula, each element of $(A+\mathrm{i k }B)^{-1}$ is the quotient of two polynomials, except for a finite number of complex values $\mathrm{k}$, i.e. the complex zeros of $\det(A+\mathrm{ik}B)$. It follows that for any $1\leq j,j'\leq n$, there exist two  
complex polynomials $p_{jj'},q_{jj'} \in \mathbb{C}[X]$ such that, excepting a finite number o values of $\mathrm{k}$, $G_{j,j'}(\mathrm{k},A,B)$ is the quotient of the two polynomials, $p_{jj'}(\mathrm{k}),q_{jj'}(\mathrm{k})$.  Since $\det(A+\mathrm{i k} B)$ is a polynomial in $\mathrm{k}$ with zeros on the imaginary axis it follows that the following holds   in a strip around the real axis 
 \[
 G_{j,j'}(z,A,B)=\dfrac{p_{jj'}(z)}{q_{jj'}(z)}, \quad |\Im z|<\rho, {z}\neq 0.
 \]
Moreover
 $G(\mathrm{k},A,B)$ is a unitary matrix for any $\mathrm{k} \in \mathbb{R} \setminus \{0\}$ (due to \cite[Proposition 3.7]{inverse}). It follows that the above polynomials satisfy for any $\mathrm{k}\in \rr\setminus\{0\}$
$$\mathlarger{\sum_{l=1}^{n}} \dfrac{p_{jl}(\mathrm{k})}{q_{jl}(\mathrm{k})} \ \dfrac{\overline{p_{j'l}(\mathrm{k})}}{\overline{q_{j'l}(\mathrm{k})}} = \delta_{jj'}, \quad \ j,j'=1,\dots,n.$$
 Consequently, all the involved quotients satisfy
\begin{equation}
\bigg| \dfrac{p_{jj'}(\mathrm{k})}{q_{jj'}(\mathrm{k})}\bigg| \leq 1, \quad j,j'=1,\dots,n, \ k \in \mathbb{R} \setminus \{ 0 \}.\label{bound}
\end{equation}
This implies in particular that the degrees of the involved polynomials satisfy $\deg(p_{jj'}) \leq \deg(q_{jj'})$ for any $1\leq j, j'\leq n$. In order to control the behavior of $G_{j,j'}(z)$ near $z=0$ 
we decompose the above polynomials as
\[
p_{jj'}(\mathrm{k})=\tilde{p}_{jj'} (\mathrm{k})p^*_{jj'}(\mathrm{k}), \quad q_{jj'}(\mathrm{k})=\tilde{q}_{jj'} (\mathrm{k})q^*_{jj'}(\mathrm{k}), 
\]
where $\tilde{p}_{jj'}$ and $\tilde{q}_{jj'}$ denote the monomials with the largest degree such that  $p^*_{jj'}$ and $q^*_{jj'}$ are polynomials with non-zero constant term. 
Using again property \eqref{bound} for $\mathrm{k}\sim 0$ we obtain that the degrees of the new polynomials can also be compared, more precisely
$\deg(\tilde{p}_{jj'}) \leq \deg(\tilde{q}_{jj'})$.
The above properties of the involved polynomials imply that for any $1\leq j,j'\leq n$ the following holds
\begin{equation}
	\label{est.G}
	 \sup _{|\Im z|\leq \rho/2}|G_{j,j'}(z,A,B)|\leq C(A,B)<\infty.
\end{equation}

%
%
%
%
%

%
 We now express the above estimate in terms of $\mathrm{k}$ and $\delta$ by choosing $z=\sqrt{\mathrm{k}^2\pm i\delta}$.
For all  real numbers $\mathrm{k}$ and $\delta \geq0$, the following hold
\begin{equation}\label{bound_delta}
	 0\leq\Im\sqrt{\mathrm{k}^2 \pm \mathrm{i} \delta} \leq (\delta /2)^{1/2},
\end{equation}
\begin{equation*}
	\label{real.part}
	|\mathrm{k}|\leq \Re\sqrt{\mathrm{k}^2+ i\delta}\leq \sqrt{\mathrm{k}^2+\frac \delta 2},
\end{equation*}
\begin{equation*}
	\label{real.part.2}
	-\sqrt{\mathrm{k}^2+\frac \delta 2} \leq \Re\sqrt{\mathrm{k}^2- i\delta}\leq   -|\mathrm{k}|.
\end{equation*}
We choose $\delta_0=2(|\rho|/2)^2$ such that $(\delta_0 /2)^{1/2}=\rho/2$.
Let us fix $\delta\leq \delta_0$. Using \eqref{bound_delta} and \eqref{est.G} we obtain that $G_{jj'}(\sqrt{\mathrm{k}^2 \pm \mathrm{i} \delta},A,B)$ is uniformly bounded for $0\leq \delta\leq \delta_0$ and $\mathrm{k}\in \rr$.
In view of representation \eqref{rep.r} we obtain the desired estimate \eqref{bound_strip}.

%
%
%
%

\medskip

Let us fix two points on $\mathcal{G}$, $x=(x_1,\dots, x_n)$ and $y=(y_1,\dots, y_n)$, $x_i\geq 0$,  $y_i\geq 0$ for all $i=\{1,\dots, n\}$, and choose $k\in\rr\setminus\{0\}$. Notice that the matrix $G(|\mathrm{k}|,A,B)$ is well defined. 
 Since $\sqrt{\mathrm{k}^2 \pm \mathrm{i} \delta} \to  \pm|k| $ as $\delta >0$ tends to zero and the composing terms are continuous as function of $\mathrm{k}$ it is immediate that
$$ \lim_{\delta \to 0} \mathrm{k} \mathrm{r}(x,y,\sqrt{\mathrm{k}^2 \pm \mathrm{i} \delta}) = \mathrm{k} \mathrm{r}(x,y,\mathrm{|k|}), \ \forall \mathrm{k} \neq 0.$$
The proof is now finished.
\end{proof}

\medskip

We denote by $C_0^{\infty}(\mathcal{G})$ the set of all functions $f=(f_j)^T_{j=\overline{1,n}}$ such that each $f_j$ belongs to the space of compactly supported infinitely differentiable functions on the interior of the edge $\mathring{I_j}=(0,\infty)$.

\begin{prop} \label{explicitsol}

For every test function $u_0 \in C_0^{\infty}(\mathcal{G})$, the solution for the linear Schr\"{o}dinger equation \eqref{system star} can be written as
\begin{equation}
\mathrm{e}^{\mathrm{i}t \Delta(A,B)}\mathcal{P}_{ac}(-\Delta(A,B))u_0(x)= \lim_{\epsilon \searrow 0} \ \dfrac{1}{\pi \mathrm{i}} \int_{\mathcal{G}} u_0(y) \int_{\mathbb{R}} \mathrm{e}^{-(\mathrm{i}t+\epsilon) \mathrm{k}^2}   \mathrm{k} \mathrm{r}(x,y,\mathrm{k}) \ d\mathrm{k} \ dy, \label{solution}
\end{equation}

\noindent where the limit is taken in $L^2(\mathcal{G})$.\\

\end{prop}

\begin{proof} In the sequel, for the ease of notation, we may refer to $-\Delta(A,B)$ also as $H$.
Let $u_0 \in C_0^{\infty}(\mathcal{G})$. By \cite[Theorem 3.1]{teschl}, it holds that
\begin{equation*}
\mathrm{e}^{-\mathrm{i}tH}u_0=\lim_{\epsilon \searrow 0} \ \mathrm{e}^{-(\mathrm{i}t + \epsilon)H }u_0 \quad in \ L^2(\mathcal{G}).
\end{equation*}
 Since the absolutely continuous spectrum of $H$ is the interval $[0, \infty)$, by the spectral theorem of representation via the resolvent for bounded functions of unbounded self-adjoint operators \cite[Theorem XII.2.11]{dunford}, we have 
\begin{equation*}
\mathrm{e}^{-(\mathrm{i}t + \epsilon) H} \mathcal{P}_{ac}(H)u_0 (x)= \lim_{\delta \to 0} \ \frac{1}{2 \pi \mathrm{i}} \int_0^{\infty} \mathrm{e}^{-(\mathrm{i}t + \epsilon)\lambda} \ \Big[ R_{\lambda +\mathrm{i} \delta} - R_{\lambda - \mathrm{i} \delta} \Big] \ u_0(x) \ d\lambda,
\end{equation*}
where $R_z=(H-z)^{-1}$.
Hence, using the resolvent representation in Lemma \ref{lemma.kernel} together with with {\it the dominated convergence theorem},
\begin{align*}
\mathrm{e}^{-(\mathrm{i}t + \epsilon) H} &\mathcal{P}_{ac}(H) u_0 (x)= \\
&= \lim_{\delta \to 0} \ \frac{1}{2 \pi \mathrm{i}} \int_0^{\infty} \mathrm{e}^{-(\mathrm{i}t + \epsilon)\lambda} \ \int_{\mathcal{G}} \big[ \mathrm{r}(x,y, \sqrt{\lambda + \mathrm{i} \delta}) \ - \ \mathrm{r}(x,y,\sqrt{\lambda - \mathrm{i} \delta})\  \big] \  u_0(y) \ dy \ d \lambda \\
&= \frac{1}{2 \pi \mathrm{i}} \int_0^{\infty} \mathrm{e}^{-(\mathrm{i}t + \epsilon)\lambda} \ \lim_{\delta \to 0} \int_{\mathcal{G}} \big[ \mathrm{r}(x,y, \sqrt{\lambda + \mathrm{i} \delta}) \ - \ \mathrm{r}(x,y,\sqrt{\lambda - \mathrm{i} \delta})\  \big] \  u_0(y) \ dy \ d \lambda.
\end{align*}
 Making now the change of variables $\mathrm{k}=\sqrt{\lambda}$ in the integral containing $\mathrm{r}(x,y,\sqrt{\lambda + \mathrm{i} \delta})$ and $\mathrm{k}=-\sqrt{\lambda}$ in the integral containing $\mathrm{r}(x,y,\sqrt{\lambda - \mathrm{i} \delta})$ leads to
\begin{multline*}
\mathrm{e}^{-(\mathrm{i}t + \epsilon) H} \mathcal{P}_{ac}(H)u_0 (x)= \frac{1}{\pi \mathrm{i}} \int_{-\infty}^0 \mathrm{e}^{-(\mathrm{i}t + \epsilon)\mathrm{k}^2} \ \lim_{\delta \searrow 0} \int_{\mathcal{G}} \ \mathrm{k}\mathrm{r}(x,y, \sqrt{\mathrm{k}^2 - \mathrm{i} \delta}) \  u_0(y) \ dy \ d\mathrm{k} \\
+ \frac{1}{\pi \mathrm{i}} \int_0^{\infty} \mathrm{e}^{-(\mathrm{i}t + \epsilon)\mathrm{k}^2} \ \lim_{\delta \searrow 0} \int_{\mathcal{G}} \ \mathrm{k}\mathrm{r}(x,y, \sqrt{\mathrm{k}^2 + \mathrm{i} \delta}) \  u_0(y) \ dy \ d\mathrm{k}.
\end{multline*}
Using again {\it the dominated convergence} and {\it Fubini's} theorems, together with Lemma \ref{uniform} we finally obtain
\begin{align*}
&\mathrm{e}^{-(\mathrm{i}t+ \epsilon) H}\mathcal{P}_{ac}(H) u_0(x) =\\
&=\frac{1}{\pi \mathrm{i}} \int_{-\infty}^0 \mathrm{e}^{-(\mathrm{i}t + \epsilon)\mathrm{k}^2} \int_{\mathcal{G}} \ \mathrm{k}\mathrm{r}(x,y, -|\mathrm{k}|)   u_0(y)  dy  d\mathrm{k} + \frac{1}{\pi \mathrm{i}} \int_0^{\infty} \mathrm{e}^{-(\mathrm{i}t + \epsilon)\mathrm{k}^2}\int_{\mathcal{G}} \ \mathrm{k}\mathrm{r}(x,y, |\mathrm{k}| )  u_0(y)  dy d\mathrm{k}\\
&=\frac{1}{\pi \mathrm{i}} \int_{\mathcal{G}} u_0(y) \int_{\mathbb{R}} \mathrm{e}^{-(\mathrm{i}t+\epsilon) k^2}   \mathrm{k} \mathrm{r}(x,y,\mathrm{k}) \ d\mathrm{k} \ dy .
\end{align*}

\noindent Passing now to the limit $\epsilon \to 0$ in $L^2(\mathcal{G})$, we obtain the desired representation

\begin{equation*}
\mathrm{e}^{-\mathrm{i}t H} \mathcal{P}_{ac}(H) u_0(x)= \lim_{\epsilon \searrow 0} \ \dfrac{1}{\pi \mathrm{i}} \int_{\mathcal{G}} u_0(y) \int_{\mathbb{R}} \mathrm{e}^{-(\mathrm{i}t+\epsilon) \mathrm{k}^2}   \mathrm{k} \mathrm{r}(x,y,\mathrm{k}) \ d\mathrm{k} \ dy,
\end{equation*}

\noindent which completes the proof.
\end{proof}


\subsection{Dispersive and Strichartz estimates}

Having the explicit form of the solution expressed in Proposition \ref{explicitsol}, we obtain first the $L^{\infty}-$time decay of the solution, using a classical result for oscillatory integrals and, as a consequence, $L^p \to L^{p'}$ dispersive properties. Together with the properties of the unitary group and the projection operator, the Strichartz estimates follow using a more general result of Keel and Tao.\\

\begin{proof}[Proof of Theorem \ref{theorem 1}] We first establish the $L^p-L^{p'}$ decay and then the Strichartz estimates.

\noindent {\bf Step I. ${L^p \to L^{p'}}$  estimates.} Let us first establish the $L^1-L^{\infty}$  estimate. 
 By density, it is sufficient to consider $u_0 \in C_0^{\infty}(\mathcal{G})$. In view of \corrected{Proposition \ref{explicitsol}}, each component $j=1,\dots,n$ of the solution \eqref{solution} can be rewritten as
\begin{equation}
\Big( \mathrm{e}^{-\mathrm{i}tH} \mathcal{P}_{ac}(H)u_0(x) \Big) _j= \lim_{\epsilon \searrow 0} \ \dfrac{1}{\pi \mathrm{i}} \sum_{j'=1}^{n} \int_{I_{j'}} \int_{\mathbb{R}} \mathrm{e}^{-(\mathrm{i} t+\epsilon) \mathrm{k}^2} (\mathrm{k} \mathrm{r}(x,y,\mathrm{k}))_{j,j'} \ d\mathrm{k} \ u_{0j'} (y) dy, \label{component}
\end{equation}

\noindent with
\begin{equation*}
\mathrm{k} \mathrm{r}(x,y,\mathrm{k})= \dfrac{\mathrm{i}}{2} \mathrm{diag}(\mathrm{e}^{\mathrm{i} \mathrm{k} | x_l-y_l| })_{l=\overline{1,n}} + \dfrac{\mathrm{i}}{2} \ \mathrm{diag}(\mathrm{e}^{\mathrm{i} \mathrm{k} x_l})_{l=\overline{1,n}} \ G (\mathrm{k},A,B) \ \mathrm{diag}(\mathrm{e}^{\mathrm{i} \mathrm{k} y_l})_{l=\overline{1,n}}.
\end{equation*}
Estimate \eqref{bound} and  the properties on the degrees of the polynomials we obtained in the proof of Lemma \ref{uniform} show that there exists a positive constant $C(A,B)$ such that 
 for all $ i,j=1,\dots,n$:
\begin{equation*}\label{gnorm}
\sup_{\mathrm{k}\in \rr} |G_{i,j}(\mathrm{k},A,B)|\leq 1, \quad \int _{\rr} |G'_{i,j}(\mathrm{k},A,B)| \ d\mathrm{k} \leq C(A,B) < \infty.
\end{equation*}

\noindent Since each component \eqref{component}  is the limit of a sum of $n$ integrals with respect to $y \in \mathcal{G}$, with oscillatory integrals with respect to $\mathrm{k} \in \mathbb{R}$ as integrands, applying Van Der Corput Lemma (e.g. \cite[p.~334]{stein}), for each $j=1,\dots,n$ follows that

\begin{align*}
	\Big\| \Big( \mathrm{e}^{-\mathrm{i}tH} \mathcal{P}_{ac}(H)u_0(x) \Big) _j \Big\|_{L^{\infty}(I_j)} &\leq \dfrac{1}{\sqrt{t}} \sum_{j'=1}^{n} \Big( ||\mathrm{e}^{-\epsilon \mathrm{k}^2} G_{j,j'} ||_{L^{\infty}(\mathbb{R})} + ||(\mathrm{e}^{-\epsilon \mathrm{k}^2}G_{j,j'})' ||_{L^{1}(\mathbb{R})} \Big) \  ||u_{0j'} ||_{L^1(I_{j'})}\\
	&\leq \dfrac{1}{\sqrt{t}} \sum_{j'=1}^{n} \Big( ||G_{j,j'} ||_{L^{\infty}(\mathbb{R})} + ||G'_{j,j'} ||_{L^{1}(\mathbb{R})} \Big) \  ||u_{0j'} ||_{L^1(I_{j'})}\\
	&\leq C(A,B) \dfrac{1}{\sqrt{t}} \sum_{j'=1}^{n}  \  ||u_{0j'} ||_{L^1(I_{j'})} = C(A,B) \dfrac{1}{\sqrt{t}} \|u_0 \|_{L^1(\mathcal{G})}.
\end{align*}
Since $(\mathrm{e}^{-\mathrm{i}tH})_{t \in \mathbb{R}}$ is a family of unitary operators preserving the $L^2-$norm we obtain 
\begin{equation*}
\big\| \mathrm{e}^{-\mathrm{i}tH}\mathcal{P}_{ac}(H) u_0 \big\|_{L^2(\mathcal{G})} \leq || u_0 ||_{L^2(\mathcal{G})}.
\end{equation*}
By \textit{Riesz-Thorin's interpolation theorem}, one has that for all $p \in[1,2]$ and $t \neq 0$,
\begin{equation*}
\big\| \mathrm{e}^{-\mathrm{i}tH}\mathcal{P}_{ac}(H) u_0 \big\|_{L^{p'}(\mathcal{G})} \leq C(A,B,p) |t|^{-\frac{1}{p}+\frac{1}{2}} \ ||u_0 ||_{L^p(\mathcal{G})}.
\end{equation*}
{\bf Step II. Strichartz estimates.} We employ here a result of Keel and Tao in \cite{keel}, concerning Strichartz estimates in the following abstract setting: let $(X,dx)$ be a measure space and $\mathcal{H}$ a Hilbert space. Suppose that for each time $t \in \mathbb{R}$, we have an operator $U(t):\mathcal{H} \to L^2(X)$ which obeys estimates:

\begin{enumerate}[label=(\roman*)]
\setlength\itemsep{8pt}

\item For all $t$ and $f \in \mathcal{H}$,
\begin{equation*}\label{first}
 ||U(t)f ||_{L^2(X)} \lesssim ||f ||_{\mathcal{H}}; 
\end{equation*}

\item For some $\sigma > 0$, for all $t \neq s$ and all $g \in L^1(X)$
\begin{equation*}\label{second}
|| U(t)U^*(s) g ||_{L^{\infty} (X)} \lesssim |t-s|^{-\sigma} \ || g||_{L^1(X)}.
\end{equation*}

\end{enumerate}

\noindent We recall that an exponent pair $(q, r)$ is called sharp $\sigma-$admissible if $q, r \geq 2$,
$(q, r, \sigma) \neq (2, \infty, 1)$ and
\begin{equation}
\dfrac{1}{q}+\dfrac{\sigma}{r} = \dfrac{\sigma}{2}. \label{sharp}
\end{equation}

\begin{thm}\cite{keel} \label{keel} If $U(t)$ obeys \ref{first} and \ref{second}, then the estimates 

\begin{enumerate}[label=(\Roman*)]
\setlength\itemsep{8pt}

\item $\qquad \qquad \qquad \qquad \qquad \qquad \big\| U(t)f \big\|_{{L^{q}_t} {L^{r}_x}} \lesssim || f ||_{\mathcal{H}}$ \label{e1}

\item $\qquad \qquad \qquad \qquad \qquad \bigg\| \mathlarger{\int_{\mathbb{R}}} U(t) U^*(s) F(s,\cdot) ds \bigg\|_{\mathcal{H}} \lesssim || F ||_{{L^{q'}_t} {L^{r'}_x}}$ \label{e2}

\item $\qquad \qquad \qquad \qquad \quad \bigg\| \mathlarger{\int_{s<t}} U(t) U^*(s) F(s,\cdot) ds\bigg\|_{L^q_t L^r_x} \lesssim || F ||_{{L^{\tilde{q}'}_t} {L^{\tilde{r}'}_x}}$ \label{e3}

\end{enumerate}
hold for all sharp $\sigma$-admissible exponent pairs $(q, r)$, $(\tilde{q}, \tilde{r}).$ 
\end{thm}

\noindent From the properties of the unitary group $\mathrm{e}^{-\mathrm{i}tH}$, the projection operator $\mathcal{P}_{ac}(H)$ (mass conservation, self-adjointness and the \corrected{commutative} property in \cite[Corollary XII.2.7]{dunford} and the dispersive estimate \ref{estimate 1} in Theorem \ref{theorem 1}, conditions \ref{first} and \ref{second}, are satisfied for $X=\mathcal{G}$, $\mathcal{H}=L^2(\mathcal{G})$, $U(t)=\mathrm{e}^{-\mathrm{i}tH}\mathcal{P}_{ac}(H)$ and $\sigma=\dfrac{1}{2}$, for $t \in \mathbb{R}$. Thus, by Theorem \ref{keel}, the Strichartz estimates in Theorem \ref{theorem 1} hold.

Hence, the proof of Theorem \ref{theorem 1} is now complete.
\end{proof}

\begin{proof}[Proof of Corollary A\ref{corollaryA1}]
We know by \cite[Theorem 3.7]{laplacians} that in the case of a star-shaped metric graph the Hamiltonian $H=-\Delta(A,B)$ has at most a finite number of negative eigenvalues. More precisely, it is equal with $n_{+}(AB^{\dag})$, i.e. the number of positive eigenvalues of the matrix $A B^{\dag}$. For every $j \in \{ 1,\dots, n_+(AB^{\dag}) \}$  the eigenfunction $\varphi_j$ corresponding to the eigenvalue $\lambda_j=-k_j^2$, $k_j>0$, is of type $ c_je^{-k_jx}$ on each edge parametrized by $x\in (0,\infty)$. Hence, 
\begin{equation*}
\varphi_j \in L^r(\mathcal{G}), \qquad \forall \ 1 \leq r \leq \infty.
\end{equation*}
For every $f \in L^2(\mathcal{G})$ the projection $\mathcal{P}_p(H)$ on the closed subspace spanned by the corresponding eigenfunctions $\{\varphi_j\}_{j=1}^{n_+(AB^{\dag})}$ is given by
$$ \mathcal{P}_p(H) f = \sum_{j=1}^{n_+(AB^{\dag})} <f,\varphi_j>  \varphi_j  .$$ 
%
%
and
\begin{equation*}
\mathrm{e}^{-\mathrm{i}tH} \mathcal{P}_p(H)f= \sum_{j=1}^{n_+(AB^{\dag})} \mathrm{e}^{-\mathrm{i}\lambda_j t} <f,\varphi_j> \varphi_j.
\end{equation*}
Using that   $\lambda_j\in \rr$  and {\it H\"{o}lder's inequality}, we  obtain that for every $\alpha \geq 1$,
\begin{align*}
 \| \mathrm{e}^{-\mathrm{i}tH} \mathcal{P}_p(H) f \|_{L^{r}(\mathcal{G})}& \leq \sum_{j=1}^{n_+(AB^{\dag})} |<f,\varphi_j>| \| \varphi_j \|_{L^{r}(\mathcal{G})}\\
&\leq \| f \|_{L^{\alpha}(\mathcal{G})} \sum_{j=1}^{n_+(AB^{\dag})} \| \varphi_j\|_{L^{\alpha'}(\mathcal{G})} \| \varphi_j \|_{L^{r}(\mathcal{G})}\\
&\leq C(\mathcal{G},A,B,\alpha,r) \| f \|_{L^{\alpha}(\mathcal{G})}.
\end{align*}
 Taking now the $L^q$--norm on the time interval $(-T,T)$, we get that for every $\alpha \geq 1$,
\begin{equation}\label{strichartzpp1}
\| \mathrm{e}^{-\mathrm{i}tH} \mathcal{P}_p(H) f\|_{L^q((-T,T), L^r(\mathcal{G}))} \leq C(\mathcal{G},A,B,\alpha,r) (2T)^{1/q} \| f \|_{L^{\alpha}(\mathcal{G})}.
\end{equation}
%
Hence, we consequently obtain that

\zeromathleft
\begin{align}
\label{strichartzpp2} \bigg\| \int_0^t \mathrm{e}^{-\mathrm{i}(t-s)H} \mathcal{P}_p(H) F(s) \ ds \bigg\|_{L^q((-T,T), L^r(\mathcal{G}))} &\leq \bigg\| \int_0^t \| \mathrm{e}^{-\mathrm{i}(t-s)H} \mathcal{P}_p(H) F(s) \|_{L^r(\mathcal{G})} \ ds \bigg\|_{L^q(-T,T)}  \nonumber \\
&\leq        C(\mathcal{G},A,B,\alpha,r) (2T)^{1/q} \int_{-T}^T \| F(s) \|_{L^{\alpha}(\mathcal{G})} \ ds. 
\end{align}
Taking now into account that
\mathcenter
\begin{equation*}
\mathrm{e}^{-\mathrm{i}tH} u_0= \mathrm{e}^{-\mathrm{i}tH} \mathcal{P}_{ac}(H) u_0 + \mathrm{e}^{-\mathrm{i}tH} \mathcal{P}_p(H)u_0,
\end{equation*} 
the desired result follows after corroborating the estimates \eqref{strichartzpp1} and \eqref{strichartzpp2} with Theorem \ref{theorem 1}.
\end{proof}

\subsection{Well-posedness \corrected{in ${L^2(\mathcal{G})}$} of the nonlinear Schr\"{o}dinger equation} In this subsection, making use of the Strichartz estimates, we show that the problem \eqref{system nonlinear} is globally well-posed in the mild form. First we establish local in time existence and uniqueness, then, once the conservation of the $L^2(\mathcal{G})-$norm is obtained, we prove that the solution is global.

\begin{proof}[Proof of Theorem \ref{theorem 2}] We divide the proof in few steps.

 {\bf Step I. Local well-posedness.} We consider the integral equation:

\begin{equation}\label{integral eq}
u(t)=\underbrace{\mathrm{e}^{-\mathrm{i}t H} u_0 + \mathrm{i} \, \corrected{\lambda} \int_0^t \mathrm{e}^{-\mathrm{i}(t-s) H} ( | u |^{p-1} u)(s)ds }_{\mathlarger{\corrected{\Phi(u)(t)}}}
\end{equation}

%
%
%
Let us set the admissible pair $(q_0,r_0)=(4(p+1)/(p-1),p+1)$. The proof 
follows the same steps as the ones in \cite[Theorem 5.2]{linares}, more precisely, using the Strichartz estimates in Corollary A\ref{corollaryA1}, one shows that $\Phi$ is a strict-contraction on some balls of the space $C([-T,T]: L^2(\mathcal{G}))\cap L^{q_0}([-T,T], L^{r_0}(\mathcal{G}))$, for 
 some convenient $T<1$  depending on $\| u_0 \|_{L^2(\mathcal{G})} $, \corrected{$\lambda$} and $ p$. In particular, we use that the nonlinearity $g(u)=\corrected{\lambda}|u|^{p-1}u$ satisfies the following estimate when $p\in (1,5)$ (see  \cite[Theorem~5.2]{linares})
 \corrected{
\begin{equation}\label{ineg.neliniara}
\begin{aligned}
\| g(v) &- g(w) \|_{L^{q_0'}((-T,T),L^{r_0'}(\mathcal{G}))} \\
&\leq c |\lambda| T^{(5-p)/4} \, \| v-w\|_{L^{q_0}((-T,T),L^{r_0}(\mathcal{G}))} \Big(\| v\|^{p-1}_{L^{q_0}((-T,T),L^{r_0}(\mathcal{G}))} + \| w \|^{p-1}_{L^{q_0}((-T,T),L^{r_0}(\mathcal{G}))}\Big).
\end{aligned}
\end{equation}}



{\bf Step II. Extra-integrability.} We now prove that 
the solution $u$ obtained above belongs to $ L^q([-T,T], L^r(\mathcal{G}))$ for all sharp ${1}/{2}-$admissible exponent pairs $(q,r)$.
Due to the Strichartz estimates in Corollary A\ref{corollaryA1}, we have that
\begin{align*}
 \| u  &\|_{L^q([-T,T],L^r(\mathcal{G}))} \lesssim \| \mathrm{e}^{-\mathrm{i} t H} \|_{L^q([-T,T],L^r(\mathcal{G}))} + \corrected{|\lambda|} \Big\|   \int_0^t \mathrm{e}^{-\mathrm{i}(t-s)H} ( | u |^{p-1} u)(s)ds   \Big\|_{L^q([-T,T],L^r(\mathcal{G}))} \\
&\lesssim [1+(2T)^{1/q}] \| u_0 \|_{L^2(\mathcal{G})}+ \|  g(u) \|_{L^{ q_0'}([-T,T], L^{r_0'}(\mathcal{G}))} + (2T)^{1/q_0} || g(u) ||_{L^{1}((-T,T)), L^{r_0'}(\mathcal{G}))}\\
&\lesssim [1+(2T)^{1/q}] \| u_0 \|_{L^2(\mathcal{G})}+ [1+(2T)^{1/q+1/q_0}]   \|  g(u) \|_{L^{ q_0'}([-T,T], L^{r_0'}(\mathcal{G}))}\\
&\lesssim [1+(2T)^{1/q}] \| u_0 \|_{L^2(\mathcal{G})}+ \corrected{|\lambda|}C(T)\|u\|_{L^{q_0}([-T,T], L^{r_0}(\mathcal{G}))}.
\end{align*}
\corrected{This shows that once the local well-posedness gives that the solution $u$ is in $L^{q_0}_{loc}(\rr, L^{r_0}(\mathcal{G}))$ then it belongs to the whole class of spaces $L^{q}_{loc}(\rr, L^{r}(\mathcal{G}))$, with $(q,r)$ being any $1/2$ sharp admissible pairs.
Moreover, when $p\leq 5$ for any admissible pair $(q,r)$ we can find another admissible pair $(q^*,r^*)$ such that $pr'=r^*$, $pq'\leq q^*$ and then
\begin{equation}
\label{ineg.neliniara.2}
	\|g(u)\|_{L^{q'}([-T,T], L^{r'}(\mathcal{G}))}=\| u\|^p_{L^{pq'}([-T,T], L^{pr'}(\mathcal{G}))} \leq C(T)\|u\|^p_{L^{q^*}([-T,T], L^{r^*}(\mathcal{G})}.
\end{equation}}

%
%
%
%

{\bf Step III.  Conservation of the $L^2-$norm and the global well-posedness.} 
In order to prove the conservation of the $L^2$-norm we will regularize the solution obtained above and  analyze the $L^2$-norm of the regularized solution. Once the conservation of the $L^2-$norm is obtained, we can repeat the local argument in Step I as many times as we wish, preserving the length of time interval to get a global solution. The following lemma will play an important role in our approach. Its proof will be given latter.

\begin{lem} \label{regularization}
The operator $(I+\varepsilon^2 H)^{-1}$ satisfies the following:
\begin{enumerate}[label=(\roman*)] \label{regularoperator}

\item\label{est.k.1} There exist two  constants $\varepsilon_0=\varepsilon_0(A,B)$ and $C(A,B)$ such that   for any $1\leq p \leq \infty$:
\[
\big\|(I+\varepsilon^2 H)^{-1} \big\|_{\mathcal{L}(L^p(\mathcal{G}),L^p(\mathcal{G}))} \leq C(A,B), \quad  \forall \, |\varepsilon|<\varepsilon_0,
\]

\item For any $1\leq p<\infty$ and $\varphi \in L^p(\mathcal{G})$,
$$ \lim_{\varepsilon \to 0} \ (I+\varepsilon^2 H)^{-1} \varphi = \varphi \qquad in \ L^p(\mathcal{G}); $$

\item For any $(q,r)$ such that $1\leq q,r<\infty$ and  $\psi \in L^q((-T,T), L^r(\mathcal{G}))$,
$$ \lim_{\varepsilon \to 0} \ (I+\varepsilon^2 H)^{-1} \psi = \psi \qquad in \ L^q((-T,T), L^r(\mathcal{G})).$$
\end{enumerate}
\end{lem}

For any initial data $u_0 \in L^2(\mathcal{G})$ we consider the local solution obtained at Step I, 
 $u \in C([-T,T],L^2(\mathcal{G})) \underset{(q,r)-adm}{\medcap} L^q((-T,T),L^r(\mathcal{G}))$ \corrected{satisfying \eqref{integral eq}}.
 Let us consider $u_\varepsilon$ the regularization of $u$ defined as
\begin{equation*}
u_{\varepsilon}(t):= (I+\varepsilon^2 H)^{-1} u(t).
\end{equation*}
Since the operator $(I+\varepsilon^2 H)^{-1}$ commutes with the unitary group $\mathrm{e}^{-\mathrm{i}tH}$, we obtain that
\begin{equation*}
\begin{cases}
u_{\varepsilon} \in C([-T,T],\mathcal{D}(H)), \\
u_{\varepsilon}(t)=\mathrm{e}^{-\mathrm{i}tH}u_{0\varepsilon} + \mathrm{i}  \mathlarger{\int_0^t} \mathrm{e}^{-\mathrm{i}(t-s)H}(I+\varepsilon^2 H)^{-1}(|u|^{p-1}u)(s) \ ds,
\end{cases}
\end{equation*}
 where $u_{0\varepsilon}=(I+\varepsilon^2 H)^{-1} u_0$. In view of \eqref{ineg.neliniara.2} we have that $g(u) \in L^1((-T,T),L^2(\mathcal{G}))$ and hence 
$(I+\varepsilon^2 H)^{-1} g(u) \in L^1((-T,T),\mathcal{D}(H))$.
 %
%
%
%
%
%
%
%
Applying now \cite[Proposition 4.1.9]{cazenaveharraux}, we obtain that
$
u_{\varepsilon} \in C([-T,T],\mathcal{D}(H))\cap W^{1,1}((-T,T),L^2(\mathcal{G})),
$
and, moreover,
\begin{equation*}
\begin{cases}
\mathrm{i} \dfrac{\partial u_{\varepsilon}}{\partial t} (t)=H u_{\varepsilon} (t) + (I+\varepsilon^2 H)^{-1} g(u), \qquad for \, a.e. \ t \in [-T,T], \\
u_{\varepsilon}(0)=u_{0 \varepsilon}.
\end{cases}
\end{equation*}
This shows that the map $[-T,T] \ni t \longmapsto \| u_{\varepsilon}(t) \|^2_{L^2(\mathcal{G})} $ is absolutely continuous and
\begin{equation*}
\| u_{\varepsilon}(t) \|^2_{L^2(\mathcal{G})} - \| u_{0 \varepsilon} \|^2_{L^2(\mathcal{G})} = 2 \Re \ \mathrm{i} \int_0^t \int_{\mathcal{G}} (I+\varepsilon^2 H)^{-1} (g(u))(s,x) \ \overline{u}_{\varepsilon}(s,x) \ dx \ ds.
\end{equation*}
In view of Lemma \ref{regularization}, as $\varepsilon$ goes to zero,
\begin{equation*}
u_{\varepsilon}(t) \rightarrow u(t) \qquad in \ L^2(\mathcal{G}), \ \forall \ t \in  (-T,T),
\end{equation*}
 and for every $(q,r)$ sharp $1/2$--admissible the following holds
\[
u_{\varepsilon} \rightarrow u \qquad in \ L^q((-T,T),L^r(\mathcal{G})).
\]
Furthermore, taking into account \eqref{ineg.neliniara.2}, we also have that $g(u)\in L^{q'}((-T,T),L^{r'}(\mathcal{G}))$ and as $\varepsilon $ goes to zero
\begin{equation*}
(I+\varepsilon^2 H)^{-1} g(u) \rightarrow g(u) \qquad in \ L^{q'}((-T,T),L^{r'}(\mathcal{G}))
\end{equation*}
This leads to 
\begin{equation*}
\int_0^t \int_{\mathcal{G}} (I+\varepsilon^2 H)^{-1} (g(u))(s,x) \overline{u}_{\varepsilon} (s,x) \ dx \ ds \rightarrow \int_0^t \int_{\mathcal{G}} g(u(s,x)) \overline{u} (s,x) \ dx \ ds,  
\end{equation*}
as $\varepsilon$ tends to zero. Finally, since $\Re (\mathrm{i} g(u), u)=0$, we obtain the conservation of the $L^2(\mathcal{G})$--norm of $u(t)$.

The proof of Theorem \ref{theorem 2} is now complete.
\end{proof}

\begin{proof}[Proof of Lemma \ref{regularoperator} ] \corrected{Let us consider without loss generality that $\varepsilon>0$}. Using  
 the representation formula \eqref{kernel} of the integral resolvent for any $\varphi=(\varphi_j)_{j=1}^n$ we obtain the following representation
\begin{equation}\label{regrepres}
(I+\varepsilon^2 H)^{-1} \varphi = \frac{1}{2 \varepsilon} \int_{\mathcal{G}} \mathrm{diag}({e^{-\frac{|x-y|}{\varepsilon}}}) \varphi (y) \ dy + \dfrac{1}{2 \varepsilon} \int_{\mathcal{G}} \mathrm{diag} ({e^{-\frac{x}{\varepsilon}}}) G\Big(\frac{\mathrm{i}}{\varepsilon}, A, B\Big) \mathrm{diag} ({e^{-\frac{y}{\varepsilon}}}) \varphi(y) \,dy. 
\end{equation}
Using the representation of matrix $G$ obtained in the proof of Lemma \ref{uniform} as quotient of polynomials we obtain that all the components of the matrix $G(\mathrm{i}/\varepsilon,A,B)$ are uniformly bounded if $\varepsilon$ is small enough (see also  \cite[Proposition 3.11]{inverse} and the proof of \cite[Theorem~3.12]{inverse}), i.e.
\begin{equation*} 
\Big\| G\Big(\frac{\mathrm{i}}{\varepsilon},A,B \Big) \Big\| \leq \mathcal{C}(A,B) < \infty, 
\end{equation*}
for all ${1}/{\varepsilon} > \rho > \rho_0 \geq 0$, where $\rho_0$ is the radius of the disk containing all the zeros of $\det(A + i k B)$, $k \in \mathbb{C}$, which lie on the imaginary axis. Using Young's inequality and the fact that the map $K_\varepsilon(x)= e^{-\varepsilon|x|}/\varepsilon$ belongs to $L^1(\rr)$ with a norm independent of $\varepsilon$ we obtain the first estimate.

The third estimate follows by applying Lebesgue's dominated convergence theorem in the time variable and the first two properties in this lemma. It remains to concentrate on proving the second property.

We recall that in the case of an even function $K\in L^1(\rr,1+|x|^2)$ it has been proved in \cite[Lemma 2.2]{denisa} that for any $1\leq p\leq \infty$
\[
\|K_\varepsilon\ast \phi -\phi\|_{L^p(\rr)}\leq C(K)\varepsilon^2 \|\phi''\|_{L^p(\rr)}, \quad \corrected{\forall \ \phi\in W^{2,p}(\rr)}.
\]
By an approximation argument we have that for any $\phi\in L^p(\rr)$, $1\leq p<\infty$, the following holds
\begin{equation}
	\label{conv.regular}
K_\varepsilon\ast \phi \rightarrow \phi \quad \text{in}\ L^p(\rr),\ \text{as} \ \varepsilon\rightarrow 0.
\end{equation}

In order to prove the convergence in $L^p({\mathcal{G}})$ with $1\leq p<\infty$, we will treat each integral separately and
we show that 
\begin{equation*}
	T^1_\varepsilon(\varphi)= \frac{1}{2 \varepsilon} \int_{\mathcal{G}} \mathrm{diag}({e^{-\frac{|x-y|}{\varepsilon}}}) \varphi (y) \ dy\rightarrow \varphi\quad 
	\text{in}\ L^p({\mathcal{G}})
\end{equation*}
and
\begin{equation}
	\label{lim.2}
	\dfrac{1}{2 \varepsilon} \int_{\mathcal{G}} \mathrm{diag} ({e^{-\frac{x}{\varepsilon}}}) G\Big(\frac{\mathrm{i}}{\varepsilon}, A, B\Big) \mathrm{diag} ({e^{-\frac{y}{\varepsilon}}}) \varphi(y) \ dy\rightarrow 0 \quad 
	\text{in}\ L^p({\mathcal{G}}).
\end{equation}
 Note that each component of $T^1_\varepsilon(\varphi)$ is explicitly given by 
\begin{equation*}
[T^1_\varepsilon(\varphi)]_j(x_j)= \dfrac{1}{2 \varepsilon} \int_0^{\infty} {e^{-\frac{|x_j-y_j|}{\varepsilon}}} \varphi_j (y) \ dy,\ x_j\in (0,\infty), \ j\in \{1,\dots,n\}.
\end{equation*}
We will prove that $[T^1_\varepsilon(\varphi)]_j\rightarrow \varphi_j$ in $ L^p(0, \infty)$. To do that we observe that 
\[
[T^1_\varepsilon(\varphi)]_j(x)=( {K}_\varepsilon \ast \tilde{\varphi_j} ) (x), \ x\in (0,\infty),
\]
where
\begin{equation*} 
\widetilde{\varphi}_j(y)=
\left\{
\begin{array}{ll}
\varphi_j(y), & \text{if } y \geq 0, \\[8pt]
0,         & \text{otherwise}.
\end{array}
\right.
\end{equation*}
Using \eqref{conv.regular} we obtain that ${K}_{\varepsilon} \ast \tilde{\varphi}_j \rightarrow \widetilde{\varphi}_j$ in $L^p(\mathbb{R})$ as $\varepsilon\rightarrow 0$. Thus, 
\[
[T^1_\varepsilon(\varphi)]_j=(K_\varepsilon \ast \tilde\varphi) |_{x>0}\rightarrow \tilde \varphi_j|_{x>0}=\varphi_j \quad \text{in}\ L^p(0, \infty).
\]

Let us now prove \eqref{lim.2}. Since the elements of the matrix $G(i/\varepsilon,A,B)$ are uniformly bounded for small $\varepsilon$ it is sufficient to prove that  
for any $\phi \in L^p(0,\infty)$ the following holds
 \begin{equation*} 
T^2_{\varepsilon}(\phi)(x):= \dfrac{1}{2 \varepsilon} \int_0^{\infty} {e^{-\frac{x+y}{\varepsilon}}} \phi (y) \ dy \rightarrow 0 \quad in \ L^p(0,\infty).
\end{equation*}
Observe again that for any $x\in (0,\infty)$ we have  $T^2_{\varepsilon} (\phi)(x)=(K_\varepsilon\ast \tilde \phi) (x)$, where this time
\begin{equation*} 
\tilde{\phi}(y)=
\begin{cases}
0, & \text{if } y > 0, \\
\phi(-y),  & \text{otherwise}.
\end{cases}
\end{equation*}
Using again \eqref{conv.regular} we obtain that ${K}_{\varepsilon} \ast \tilde{\phi}\rightarrow \tilde{\phi}$ in $L^p(\mathbb{R})$ as $\varepsilon\rightarrow 0$. Thus, 
\[
T^2_\varepsilon(\phi)=(K_\varepsilon \ast \tilde\phi )_{x>0}\rightarrow\tilde{\phi}|_{x>0}=0 \quad \text{in}\ L^p(0, \infty).
\]
The proof is now completed.
\end{proof}

\vspace{10pt}

\subsection{\corrected{Well posedness in $D(\mathcal{E})$ of the nonlinear Schr\"{o}dinger equation}}

\corrected{In this subsection we give the proof of Theorem \ref{theorem 3}.  We show that if the initial data $u_0$ is in $D(\mathcal{E})$, the $L^2$--solution $u$ obtained in Theorem \ref{theorem 2} belongs in fact to $C(\mathbb{R}; D(\mathcal{E})) \cap C^1(\mathbb{R};D(\mathcal{E})^*)$. We recall that in the case of general coupling conditions, the form domain $D(\mathcal{E})$ is not necessarily preserved by the nonlinearity. Hence we cannot rely on a fixed point technique to show that equation \eqref{system nonlinear} has a solution in $C((-T,T), D(\mathcal{E}))$. However, when $p\in (1,5)$ we know by Theorem \ref{theorem 2} that there exists a solution $u$ in $C(\rr, L^2(\mathcal{G}))$.
To overcome this difficulty, we first regularize the nonlinearity and obtain  a global solution $u^\varepsilon \in C(\mathbb{R}; D(\mathcal{E})) \cap C^1(\mathbb{R};D(\mathcal{E})^*)$. Then we show that the sequence $u^\varepsilon$ strongly converges to the solution $u$ in $C([-T,T],L^2(\mathcal{G}))$, for any $T>0$. Finally, we prove that the limit function $u$ inherits the regularity properties of $u^\varepsilon$  as well as the conservation of the energy.} \\


\color{black} Let us now consider the regularized problem:

    \begin{equation*}
(\mathcal{P_{\varepsilon}}): 
\left\{
\begin{array}{lll}
\mathrm{i} u^{\varepsilon}_t (t,x)-H u^{\varepsilon}(t,x)+g_{\varepsilon}(u^{\varepsilon})=0, & t \neq 0, & x \in \mathcal{G}\\[5pt]
u^{\varepsilon}(0,x)=u_0(x), & & x \in \mathcal{G}
\end{array},
\right.
\end{equation*}

\vspace{5pt}

\noindent where $g(u)=\lambda |u|^{p-1}u$, with $\lambda \in \mathbb{R}$, $p>1$ and 
\begin{equation*}
g_{\varepsilon}(v)=J_{\varepsilon} g (J_{\varepsilon} v),
\end{equation*}
with $J_{\varepsilon}:=(I+\varepsilon H)^{-1}$, $\varepsilon\in (0,\varepsilon_0^2)$, $\varepsilon_0$ being as in Lemma \ref{regularoperator}. In the following we will assume this smallness condition without making it explicit each time. 
We will first prove well-posedness of the problem $(\mathcal{P}_{\varepsilon})$ in the energy space $D(\mathcal{E})$. Observe that in this case the nonlinearity $g_{\varepsilon}$ maps the functions from $D(\mathcal{E})$ to itself. This is not true in the case of $g(u)=\lambda |u|^{p-1}u$. Indeed, for $u \in D(\mathcal{E})$, we have $u \in H^1(\mathcal{G})$ so $g(u) \in H^1(\mathcal{G})$, but not necessarily in $D(\mathcal{E})$. This property is easily verified in the case of $\delta$ and $\delta'$--coupling. In the last case, the $D(\mathcal{E})=H^1(\mathcal{G})$ and the property is trivially satisfied. In the case of $\delta$--coupling, the energy space is given by $ D(\mathcal{E})= \{  u \in H^1(\mathcal{G}) \, : \, u \ continuous \ at \ x=0  \}, $ and then $g(u)$ is also continuous in $x=0$.

\begin{prop}\label{globalregularized}
Let $1<p<\infty$ and $\lambda \in \mathbb{R}$. For any $u_0 \in D(\mathcal{E})$, there exists a unique solution $u^{\varepsilon} \in C(\mathbb{R}; D(\mathcal{E})) \cap C^1(\mathbb{R};D(\mathcal{E})^*)$ of problem $(\mathcal{P}_{\varepsilon})$.

Moreover, $u^{\varepsilon}$ satisfies
\begin{enumerate}[label=(\roman*)]

\item $\| u^{\varepsilon}(t) \|_{L^2(\mathcal{G})}=\| u_0 \|_{L^2(\mathcal{G}}$,

\item $E(u^{\varepsilon}(t)):= \mathcal{E}(u^\varepsilon(t),u^{\varepsilon}(t)) - \dfrac{\lambda}{p+1} \mathlarger\int_{\mathcal{G}} |u^{\varepsilon}(t,x)|^{p+1} \, dx $ is conserved,
\end{enumerate}

\noindent  for all $t \in \mathbb{R}$.
\end{prop}

\begin{proof}
Applying \cite[Theorem 3.3.1, p.63]{cazenave} we obtain the existence of a global solution
\begin{equation*}
u^{\varepsilon} \in C(\mathbb{R};D(\mathcal{E})) \cap C^1(\mathbb{R};D(\mathcal{E})^*)
\end{equation*}
 which satisfies the conservation of the $L^2(\mathcal{G})$-norm
\begin{equation*}
\| u^{\varepsilon}(t) \|_{L^2(\mathcal{G})}=\| u_0 \|_{L^2(\mathcal{G})}
\end{equation*}
 and of the energy
\begin{equation*}
\begin{aligned}
E(u^{\varepsilon}(t)):&= \mathcal{E}(u^\varepsilon(t),u^{\varepsilon}(t)) - \dfrac{\lambda}{p+1} \int_{\mathcal{G}} |J_\varepsilon u^{\varepsilon}(t,x)|^{p+1} \, dx\\
&=\mathcal{E}(u_0,u_0) - \dfrac{\lambda}{p+1} \int_{\mathcal{G}} |J_\varepsilon u_0(x)|^{p+1} \, dx.
\end{aligned}
\end{equation*}
 This is possible since the nonlinearity $g_{\varepsilon}$ is Lipschitz continuous on bounded sets of $L^2(\mathcal{G})$. 
  By Lemma \ref{regularoperator} $(i)$ $J_\varepsilon$ maps $L^2(\mathcal{G})$ to itself uniformly with respect to $\varepsilon$.  Thus, for  any $v$ and $w$ in $L^2(\mathcal{G})$    we obtain  the following
\begin{align*}
\| g_{\varepsilon}(v) -g_{\varepsilon}(w)  \|_{L^2(\mathcal{G})}& = \| J_{\varepsilon} g (J_{\varepsilon} v) - J_{\varepsilon} g (J_{\varepsilon} w) \|_{L^2(\mathcal{G})}
\leq C(A,B) \| g(J_{\varepsilon} v) - g(J_{\varepsilon} w) \|_{L^2(\mathcal{G})}\\
& \leq C(A,B,p)  \| J_{\varepsilon} v - J_{\varepsilon} w  \|_{L^2(\mathcal{G})}  \big( \| J_{\varepsilon} v \|^{p-1}_{L^{\infty}(\mathcal{G})} + \| J_{\varepsilon} w \|^{p-1}_{L^{\infty}(\mathcal{G})}  \big).
\end{align*}
We now use that $H^1(\mathcal{G})$ is continuous embedded in $L^{\infty}(\mathcal{G})$. Since $J_{\varepsilon}v$ and $J_{\varepsilon}w$ belong to the energy space $D(\mathcal{E})$, we can apply the equivalence of norms in \eqref{equivnorms} to obtain 
\[
\| g_{\varepsilon}(v) -g_{\varepsilon}(w)  \|_{L^2(\mathcal{G})}\leq C(A,B,p) \| v-w \|_{L^2(\mathcal{G})}  \big( \| J_{\varepsilon}v \|^{p-1}_{D(\mathcal{E})} + \| J_{\varepsilon}w \|^{p-1}_{D(\mathcal{E})}  \big).
\]
We rely on the following result on $J_\varepsilon$ that we will prove latter. 
\begin{lem}\label{dualestimates}Let  $M$ as in \eqref{formnorm}. 
For any positive $\varepsilon$, with $\varepsilon<1/M$ the following holds
\begin{equation}\label{dual.inequality}
\varepsilon \| (I+\varepsilon H)^{-1} f \|_{D(\mathcal{E})} \leq   \| f \|_{D(\mathcal{E})^*}, \quad \forall f \in D(\mathcal{E})^*.
\end{equation}
\end{lem}
This gives us that 
\[
\|J_\varepsilon v\|_{D(\mathcal{E})}\leq \frac 1\varepsilon \|v\|_{D(\mathcal{E})^*}\leq \frac 1\varepsilon \|v\|_{L^2(\mathcal{G})}
\]
and
\[
\| g_{\varepsilon}(v) -g_{\varepsilon}(w)  \|_{L^2(\mathcal{G})}\leq 
\frac {C(A,B,p)}{\varepsilon^{p-1} }\| v-w  \|_{L^2(\mathcal{G})} \big( \| v \|^{p-1}_{L^2(\mathcal{G})} + \|w\|^{p-1}_{L^2(\mathcal{G})} \big).
\]
This proves that $g_\varepsilon$ is is Lipschitz continuous on bounded sets of $L^2(\mathcal{G})$ and the results in \cite[Theorem 3.3.1, p.63]{cazenave}  apply. 
 Note that since we only want to prove the existence of solutions, the dependence of the constant on $\varepsilon$ is not an issue in this case.
\end{proof}

In the following proposition we obtain uniform estimates for the sequence of solutions $(u^{\varepsilon})_{\varepsilon>0}$.
\begin{prop}\label{estimateueps} Let $u_0$ in $D(\mathcal{E})$
and $u^{\varepsilon}$ the solution of $(\mathcal{P_{\varepsilon}})$ as in Proposition \ref{globalregularized}. If one of the following conditions holds:
\begin{enumerate}
\item $\lambda <0$,
\item $\lambda > 0$ and $1<p <5$,
\end{enumerate}
 then there exists a positive constant $C(\|u_0\|_{D(\mathcal{E})})$ such that 
 for all small enough $\varepsilon>0$ 
\begin{equation}\label{estimateueps1}
\|u^{\varepsilon} \|_{C(\mathbb{R}; D(\mathcal{E}))} \leq C(\|u_0\|_{D(\mathcal{E})}).
\end{equation}

Moreover, if $1<p<5$, 
\begin{equation}\label{estimateueps2}
\| u^{\varepsilon} \|_{L^{q_0}((-T,T);L^{r_0}(\mathcal{G}))} \leq C(T,\lambda,  \| u_0 \|_{L^2(\mathcal{G})}),
\end{equation}
where $(q_0,r_0)=(4(p+1)/(p-1),p+1)$.
\end{prop}

\begin{proof}
From Proposition \ref{globalregularized} $(i)-(ii)$ and \eqref{formnorm}, we have for all $t \in \mathbb{R}$
\begin{equation*}
\| u^{\varepsilon} (t) \|^2_{D(\mathcal{E})}-\dfrac{\lambda}{p+1} \int_{\mathcal{G}} |J_\varepsilon u^{\varepsilon}(t,x)|^{p+1} \, dx =\|u_0\|^2_{D(\mathcal{E})}-\dfrac{\lambda}{p+1} \int_{\mathcal{G}} |J_\varepsilon u_0(x)|^{p+1} \, dx.
\end{equation*}
When $\lambda <0$, we easily obtain the desired estimate for $u^{\varepsilon}$, since  $J_\varepsilon$ is uniformly bounded between  $L^p(\mathcal{G})$ spaces. 
%
Consider now the case $\lambda>0$. By the one dimensional Gagliardo-Nirenberg inequality generalized to star-graphs \cite[inequality (2.3)]{adamivariational} and the conservation of the $L^2$-norm, we have that
\begin{equation}\label{gagliardonirenberg}
\begin{aligned}
\| u^{\varepsilon} (t)\|^{p+1}_{L^{p+1}(\mathcal{G})} &\leq c \| (u^{\varepsilon})' (t)\|^{(p-1)/2}_{L^2(\mathcal{G})} \, \| u^{\varepsilon} (t)\|^{(p+3)/2}_{L^2(\mathcal{G})} \leq c \| u^{\varepsilon}(t) \|^{(p-1)/2}_{D(\mathcal{E})} \| u_0 \|^{(p+3)/2}_{L^2(\mathcal{G})}.
\end{aligned}
\end{equation}

\noindent From the conservation of the energy $E(u^{\varepsilon}(t))$ in Proposition \ref{globalregularized} $(ii)$, we can also write that
\begin{equation*}
\begin{aligned}
\| u^{\varepsilon} (t)\|^2_{D(\mathcal{E})} &\leq \|u_0\|^2_{D(\mathcal{E})}+ \dfrac{\lambda}{p+1} \int_{\mathcal{G}} |J_{\varepsilon} u^{\varepsilon}(t)|^{p+1} \, dx\\
& \leq \|u_0\|^2_{D(\mathcal{E})} + \dfrac{\lambda C(A,B)}{p+1} \| u^{\varepsilon} (t)\|^{p+1}_{L^{p+1}(\mathcal{G})} \\
& \leq  \|u_0\|^2_{D(\mathcal{E})} + \dfrac{\lambda C(A,B)}{p+1} \| u^{\varepsilon} (t)\|^{(p-1)/2}_{D(\mathcal{E})} \, \| u_0 \|^{(p+3)/2}_{L^2(\mathcal{G})},
\end{aligned}
\end{equation*}

\noindent where the first inequality follows after applying Lemma \ref{regularization} $(i)$, and the last one is obtained via the Gagliardo-Nirenberg estimate in \eqref{gagliardonirenberg}. Since $p<5$, 
 the desired estimate in  \eqref{estimateueps1} is immediately obtained.\\

Let us now consider now the sharp $1/2$--admissible pair $(q_0,r_0)$ as above and $T<1$. By Corollary A\ref{corollaryA1} and inequality \eqref{ineg.neliniara} we obtain that for some constant $C(p)=C(q_0,r_0)$ 
\begin{equation*}
\begin{aligned}
\| u^{\varepsilon} \|_{L^{q_0}((-T,T);L^{r_0}(\mathcal{G}))} &\leq \| \mathrm{e}^{-\mathrm{i}t H} u_0 \|_{L^{q_0}((-T,T);L^{r_0}(\mathcal{G}))} +  \bigg\| \int_0^t \mathrm{e}^{-\mathrm{i}(t-s) H} g_{\varepsilon}(u^{\varepsilon}(s)) \, ds \bigg\|_{L^{q_0}((-T,T);L^{r_0}(\mathcal{G}))}\\
&\leq C(p) \| u_0 \|_{L^2(\mathcal{G})} +C(p) \| g_{\varepsilon}(u^{\varepsilon}) \|_{L^{q'_0}((-T,T);L^{r'_0}(\mathcal{G}))}\\
&\leq C(p)  \| u_0 \|_{L^2(\mathcal{G})} + |\lambda|C(p) T^{(5-p)/4} \|u^{\varepsilon} \|^p_{L^{q_0}((-T,T);L^{r_0}(\mathcal{G}))}.
\end{aligned}
\end{equation*}
Hence, there exists a time $T_0=T_0(p,\lambda,\|u_0\|_{L^2(\mathcal{G})})\simeq \min\{ \|u_0\|_{L^2(\mathcal{G})}^{-4(p-1)/(5-p)}, 1\}$, independent of $\varepsilon$, such that for any $T<T_0$,
\begin{equation*}
\| u^{\varepsilon}(t) \|_{L^{q_0}((-T,T);L^{r_0}(\mathcal{G})} \leq  2 C(p)  \| u_0 \|_{L^2(\mathcal{G})}.
\end{equation*}
Repeating the above argument on any time interval of length $2T_0$ and using the conservation of the $L^2$-norm of the solution we obtain the desired estimate.
\end{proof}

Let us recall that for any initial data $u_0$ in $L^2(\mathcal{G})$ and $1<p<5$, by Theorem \ref{theorem 2} we have that there exists a unique solution $u \in C(\mathbb{R};L^2(\mathcal{G}))\cap L^{q_0}_{loc}(\rr, L^{r_0}(\mathcal{G}))$ of problem \eqref{system nonlinear}. In particular, this holds true for $u_0 \in D(\mathcal{E})$. Let us prove now the following result.

\begin{prop}
Let $\lambda \in \mathbb{R}$ and $1<p<5$. Let $u_0 \in D(\mathcal{E})$, $u^{\varepsilon}$ the solution of problem $(\mathcal{P_{\varepsilon}})$ given by Proposition \ref{globalregularized}, and $u$ the solution of \eqref{system nonlinear} as in Theorem \ref{theorem 2}. Then for any $T>0$
\begin{equation*}
\| u^{\varepsilon} - u \|_{C([-T,T]; L^2(\mathcal{G}))} \rightarrow 0 \ as \ \varepsilon \to 0.
\end{equation*}
\end{prop}

\begin{proof}In order to simplify the proof 
we prove the estimate for any $T<1$. We use 
 the integral equations satisfied by $u^{\varepsilon}$ and $u$, respectively:
\begin{equation*}
u^{\varepsilon}(t)=\mathrm{e}^{- \mathrm{i} t H} u_0 + \mathrm{i} \lambda \int_0^t \mathrm{e}^{-\mathrm{i}(t-s) H} g_{\varepsilon}(u^{\varepsilon})(s) \, ds,
\end{equation*}

\begin{equation*}
u(t)=\mathrm{e}^{- \mathrm{i} t H} u_0 + \mathrm{i} \lambda \int_0^t \mathrm{e}^{-\mathrm{i}(t-s) H} g(u)(s) \, ds.
\end{equation*}

\noindent Take the sharp $1/2$--admissible pair $(q_0,r_0)=(4(p+1)/(p-1),p+1)$ and $(q,r)=\{(\infty,2), (q_0,r_0)\}$. 
By Corollary A\ref{corollaryA1}  with $T<1$:
\begin{align*}
\| u^{\varepsilon} -u &\|_{L^{q}((-T,T),L^{r}(\mathcal{G})}  = \bigg\| \lambda \int_0^t \mathrm{e}^{-\mathrm{i}(t-s) H} \big[ g_{\varepsilon}(u^{\varepsilon})(s) - g(u)(s)  \big] \, ds \bigg\|_{L^{q}((-T,T),L^{r}(\mathcal{G}))}\\
& \leq C \, |\lambda| \, \| g_{\varepsilon}(u^{\varepsilon}) - g(u) \|_{L^{q_0'}((-T,T),L^{r_0'}(\mathcal{G}))}= c \, |\lambda| \, \| J_{\varepsilon} g (J_{\varepsilon}  u^{\varepsilon}) - g(u)  \|_{L^{q_0'}((-T,T),L^{r_0'}(\mathcal{G}))}.
\end{align*}
By Lemma \ref{regularoperator},
\begin{align*}
 \| J_{\varepsilon}& g (J_{\varepsilon}  u^{\varepsilon}) - g(u)  \|_{L^{q_0'}((-T,T),L^{r_0'}(\mathcal{G}))}\\
& \leq  \| J_{\varepsilon} g(J_{\varepsilon} u^{\varepsilon}) - J_{\varepsilon} g(J_{\varepsilon} u)  \|_{L^{q_0'}((-T,T),L^{r_0'}(\mathcal{G}))}\\
& \qquad \qquad + \| J_{\varepsilon} g(J_{\varepsilon} u) - J_{\varepsilon} g(u)  \|_{L^{q_0'}((-T,T),L^{r_0'}(\mathcal{G}))}+ \| J_{\varepsilon} g(u) - g(u) \|_{L^{q_0'}((-T,T),L^{r_0'}(\mathcal{G}))} \\
& \lesssim  \| g(J_{\varepsilon} u^{\varepsilon}) - g(J_{\varepsilon} u)  \|_{L^{q_0'}((-T,T),L^{r_0'}(\mathcal{G}))}\\
& \qquad \qquad + \|  g(J_{\varepsilon} u) -  g(u)  \|_{L^{q_0'}((-T,T),L^{r_0'}(\mathcal{G}))}+ \| (J_{\varepsilon}-I) g(u) \|_{L^{q_0'}((-T,T),L^{r_0'}(\mathcal{G}))},
\end{align*}
 Let us now treat separately each one of the terms in the right hand side. Using inequality \eqref{ineg.neliniara} 
with $v=J_\varepsilon u^{\varepsilon}$ and $w=J_{\varepsilon} u$, by Lemma \ref{regularoperator} and Corollary A\ref{corollaryA1} we consequently obtain that
\begin{align*}
\| &g(J_\varepsilon u^{\varepsilon}) - g(J_{\varepsilon} u) \|_{L^{q_0'}((-T,T),L^{r_0'}(\mathcal{G}))}\\
 &\leq c |\lambda| T^{\frac{5-p}4}  \| J_\varepsilon u^{\varepsilon}-J_{\varepsilon} u\|_{L^{q_0}((-T,T),L^{r_0}(\mathcal{G}))}  \bigg(\| J_\varepsilon u^{\varepsilon}\|^{p-1}_{L^{q_0}((-T,T),L^{r_0}(\mathcal{G}))} + \| J_{\varepsilon} u\|^{p-1}_{L^{q_0}((-T,T),L^{r_0}(\mathcal{G}))}\bigg)\\
& \leq c|\lambda| T^{\frac{5-p}4}  \, \|  u^{\varepsilon}- u\|_{L^{q_0}((-T,T),L^{r_0}(\mathcal{G}))}  \bigg( \|  u^{\varepsilon}\|^{p-1}_{L^{q_0}((-T,T),L^{r_0}(\mathcal{G}))} + \| u\|^{p-1}_{L^{q_0}((-T,T),L^{r_0}(\mathcal{G}))}\bigg)\\
& \leq c\big(\lambda, \| u_0 \|^{p-1}_{L^2(\mathcal{G})}\big) T^{\frac{5-p}4}\|  u^{\varepsilon}- u \|_{L^{q_0}((-T,T),L^{r_0}(\mathcal{G}))},
\end{align*}

\noindent where in the last inequality we made use of the estimate \eqref{estimateueps2} in Proposition \ref{estimateueps}, which holds also for $u$. Proceeding similarly in the case of the second term with $v=J_\varepsilon u$ and $w=u$, and using Lemma \ref{regularoperator}   with $u\in L^{q_0}((-T,T),L^{r_0}(\mathcal{G}))$
we get
\begin{align*}
\| g(J_\varepsilon u) & - g(u) \|_{L^{q_0'}((-T,T),L^{r_0'}(\mathcal{G}))}\\ &\leq c(\lambda) T^{\frac{5-p}4} \, \| J_\varepsilon u- u\|_{L^{q_0}((-T,T),L^{r_0}(\mathcal{G}))} \bigg( \| J_\varepsilon u\|^{p-1}_{L^{q_0}((-T,T),L^{r_0}(\mathcal{G}))} + \| u\|^{p-1}_{L^{q_0}((-T,T),L^{r_0}(\mathcal{G}))}\bigg)\\
& \leq c\big(\lambda, \| u_0 \|^{p-1}_{L^2(\mathcal{G})}\big)T^{\frac{5-p}4} \|   J_\varepsilon u- u \|_{L^{q_0}((-T,T),L^{r_0}(\mathcal{G}))} = o(1), \quad as \ \varepsilon \to 0.
\end{align*}
 In the case of the last term, since $g(u)\in L^{q_0'}((-T,T),L^{r_0'}(\mathcal{G}))$
  immediately follows by Lemma \ref{regularoperator} $(iii)$  that
\begin{equation*}
\| (J_{\varepsilon}-I) g(u) \|_{L^{q_0'}((-T,T),L^{r_0'}(\mathcal{G}))} = o(1), \quad as \ \varepsilon \to 0.
\end{equation*}
Let us choose $T<1$ small enough such that $c\big(\lambda, \| u_0 \|^{p-1}_{L^2(\mathcal{G})}\big) T^{(5-p)/4} < 1/4$. Thus, we can absorb the right hand side term in the left hand side and we  obtain 
\begin{equation*}
\| u^{\varepsilon} -u \|_{L^{q}((-T,T),L^{r}(\mathcal{G})} = o(1), \quad as \ \varepsilon \to 0,
\end{equation*}
 for the pairs $(q,r)\in \{(\infty,2), (q_0,r_0)\}$.
In particular, when $(q,r)=(\infty,2)$ the  regularity of $u^\varepsilon$ and $u$ gives us 
\begin{equation*}
\begin{aligned}
\| u^{\varepsilon} -u \|_{C((-T,T),L^{2}(\mathcal{G})} =o(1), \quad as \ \varepsilon \to 0, 
\end{aligned}
\end{equation*}
which completes the proof.
\end{proof}

\begin{proof}[Proof of Lemma \ref{dualestimates}]
For any $\varepsilon \in (0,1/M)$ we have
\begin{align*}
\label{}
\varepsilon\|u\|_{D(\mathcal{E})}^2&\leq (1-\varepsilon M)(u,u)+\varepsilon \|u\|_{D(\mathcal{E})}^2  \\
&= (1-\varepsilon M)(u,u)+\varepsilon ( M(u,u)+\mathcal{E}(u,u))\\
&=(u,u)+\varepsilon \mathcal{E}(u,u).
\end{align*}
This shows that for such an $\varepsilon$ we can apply {\it Lax-Milgram theorem} to obtain the existence of a unique solution
 $v$ of the equation
\begin{equation*}
\begin{cases}
v \in D(\mathcal{E}),\\
(v, \varphi) + \varepsilon \mathcal{E}(v,\varphi)=<f,\varphi>, & \forall \varphi \in D(\mathcal{E}).
\end{cases}
\end{equation*}

\noindent Choosing $\varphi=v$ in the above equality, since $\dfrac{\varepsilon M}{M+1} \leq \varepsilon \leq 1$, we obtain that
\[
 {\varepsilon} \| v \|^2_{D(\mathcal{E})}  \leq (v,v) + \varepsilon \mathcal{E}(v,v) =<f,v> \leq \| f \|_{D(\mathcal{E})^*} \|v \|_{D(\mathcal{E})}.
\]
This finishes the proof.
\end{proof}

\begin{proof}[Proof of Theorem \ref{theorem 3}]

\noindent From the estimate \eqref{estimateueps1}, we obtain that for any $t$,  up to a subsequence,
$
u^{\varepsilon} (t)\rightharpoonup v(t) $  in $D(\mathcal{E})$ and 
\[
\|v(t)\|_{D(\mathcal{E})}\leq \liminf_{\varepsilon} \| u^\varepsilon\|_{D(\mathcal{E})}\leq C(\|u_0\|_{D(\mathcal{E})}).
\]
Moreover, since $u_\varepsilon (t)\rightarrow u(t)$ in $L^2(\Gamma)$ we obtain $v(t)=u(t)$. Also, since the limit point is unique  the whole sequence $(u^\varepsilon(t))_{\varepsilon>0}$ converges to $u(t)$, not only a subsequence. Then $u\in L^\infty(\rr,D(\mathcal{E}))$ and
\begin{equation}\label{liminf}
\| u(t) \|_{D(\mathcal{E})} \leq \liminf_{\epsilon} \| u^{\varepsilon}(t) \|_{D(\mathcal{E})} \leq C(u_0),\quad \forall t\in \rr.
\end{equation}
The conservation of the $L^2$-norm of $u^\varepsilon$ and $u$ and the definition of the norm on $D(\mathcal{E})$ given in \eqref{formnorm} shows that 
\[
\mathcal{E}(u(t),u(t))\leq \liminf _{\varepsilon} \mathcal{E}(u^\varepsilon(t),u^\varepsilon(t)).
\]
For any $p>1$, since $D(\mathcal{E})\hookrightarrow L^\infty(\mathcal{G})$ we have
 the following inequality
\begin{equation*}
\| u^{\varepsilon}(t)-u(t) \|_{L^{p+1}(\mathcal{G})} \leq c(p,c_1)\| u^{\varepsilon}(t)-u(t) \|_{L^2(\mathcal{G})} \bigg( \| u^{\varepsilon}(t) \|^{p-1}_{D(\mathcal{E})} + \| u(t) \|^{p-1}_{D(\mathcal{E})} \bigg).
\end{equation*}
Corroborated with \eqref{liminf} and \eqref{estimateueps1},  we obtain the strong convergence 
	$
	u^{\varepsilon}(t) \rightarrow u(t)$,   in $ L^{p+1}(\mathcal{G})$.
 In turn, together with Lemma \ref{regularoperator}, one can write for all $t \in \mathbb{R}$,
\begin{equation*}
\begin{aligned}
\| J_{\varepsilon} u^{\varepsilon}(t) - u (t) \|_{L^{p+1}(\mathcal{G})} & \leq \| J_{\varepsilon} u^{\varepsilon}(t) - J_{\varepsilon} u(t) \|_{L^{p+1}(\mathcal{G})} + \| J_{\varepsilon} u (t) - u(t) \|_{L^{p+1}(\mathcal{G})}\\
& \lesssim \| u^{\varepsilon}(t) - u(t) \|_{L^{p+1}(\mathcal{G})} + o(1) = o(1),\ \varepsilon\rightarrow 0,
\end{aligned}
\end{equation*}
 which implies
\begin{equation}\label{Jepsp}
\int_{\mathcal{G}} |J_{\varepsilon} u^{\varepsilon}(t,x)  |^{p+1} \, dx \rightarrow \int_{\mathcal{G}} |u(t,x)  |^{p+1} \, dx \quad as \ \varepsilon \to 0.
\end{equation}

\noindent From the estimates \eqref{liminf} and \eqref{Jepsp}, taking into account the conservation of the energy in Proposition \ref{globalregularized} $(ii)$,
\begin{align*}
\mathcal{E}(u(t),u(t))-&\dfrac{\lambda}{p+1} \int_{\mathcal{G}} |u(t,x)|^{p+1} \, dx \leq \liminf _{\varepsilon}\Big[\mathcal{E}(u^{\varepsilon}(t),u^{\varepsilon}(t)) - \dfrac{\lambda}{p+1} \int_{\mathcal{G}} |J_{\varepsilon} u^{\varepsilon}(t,x)|^{p+1} \, dx\Big]\\
& = \liminf _{\varepsilon}\Big[\mathcal{E}(u_0,u_0) - \dfrac{\lambda}{p+1} \int_{\mathcal{G}} |J_\varepsilon u_0(x)|^{p+1} \, dx\Big]\\
&=\mathcal{E}(u_0,u_0) - \dfrac{\lambda}{p+1} \int_{\mathcal{G}} | u_0(x)|^{p+1} \, dx.
\end{align*}

\noindent This implies that $u \in L^{\infty}(\rr;D(\mathcal{E}))$ solution of problem \eqref{system nonlinear} satisfies
\begin{equation*}
E(u(t)) \leq E(u_0), \quad \forall\ t\in\rr .
\end{equation*}
Changing the initial time for the nonlinear equation \eqref{system nonlinear} to some $t_0\neq 0$, we obtain that 
\begin{equation*}
E(u(t)) \leq E(u(t_0)),
\end{equation*}
which shows that in fact $E(u(t))$ should be constant. 

From the inequality 
\begin{align*}
\label{}
  \| u (t)-u(s) \|_{L^{p+1}(\mathcal{G})}& \leq c(p)\| u (t)-u(s) \|_{L^2(\mathcal{G})} \big( \| u (t) \|^{p}_{D(\mathcal{E})} + \| u(s) \|^{p}_{D(\mathcal{E})}  \big)\\
  &\leq C (\|u_0\|_{D(\mathcal{E})})\| u (t)-u(s) \|_{L^2(\mathcal{G})}.
\end{align*}
and the fact that $u\in C(\rr,L^2(\mathcal{G}))$, we obtain that $u\in C(\rr,L^{p+1}(\mathcal{G}))$. Using the conservation of the energy, we get that $\mathcal{E}(u(t),u(t))\rightarrow \mathcal{E}(u(s),u(s))$ as $t\rightarrow s$. This implies immediately that $u\in C(\rr;D(\mathcal{E}))$.
Since $\Delta u\in C(\rr;D(\mathcal{E})^*)$ and $g(u)\in C(\rr; L^2(\mathcal{G}))$ we obtain that 
$iu_t=\Delta u- g(u)\in C(\rr;D(\mathcal{E})^*)$ and then $u\in C^1(\rr;D(\mathcal{E})^*)$. The proof of Theorem \ref{theorem 3} is now complete.
\end{proof}

\section*{\corrected{Acknowledgements}}

\corrected{We are grateful for the referees' valuables comments and suggestions which lead to the current improved version of this work. We would like also to thank the referee for recommending to discuss the well-posedness of the nonlinear equation in the energy domain.}



\bibliographystyle{amsplain}


\providecommand{\bysame}{\leavevmode\hbox to3em{\hrulefill}\thinspace}
\providecommand{\MR}{\relax\ifhmode\unskip\space\fi MR }
\providecommand{\MRhref}[2]{%
  \href{http://www.ams.org/mathscinet-getitem?mr=#1}{#2}
}
\providecommand{\href}[2]{#2}

\end{document}